\documentclass[11pt,letterpaper]{article}
\usepackage{macros}

\newcommand{\arrow}{\to}
\allowdisplaybreaks

\newcommand{\RemoveAuthor}{0}

\begin{document}
\title{A simple and sharper proof of the hypergraph Moore bound}
\ifnum\RemoveAuthor=1

\else
\author{Jun-Ting Hsieh\thanks{Carnegie Mellon University.  \texttt{juntingh@cs.cmu.edu}.
Supported by NSF CAREER Award \#2047933.} \and Pravesh K. Kothari\thanks{Carnegie Mellon University. \texttt{praveshk@cs.cmu.edu}.
Supported by  NSF CAREER Award \#2047933, Alfred P. Sloan Fellowship and a Google Research Scholar Award.} \and Sidhanth Mohanty\thanks{UC Berkeley.  \texttt{sidhanthm@cs.berkeley.edu}.
Supported by a Google PhD Fellowship.}}
\fi
\date{\today}
\maketitle

\begin{abstract}
The hypergraph Moore bound is an elegant statement that characterizes the extremal trade-off between the girth --- the number of hyperedges in the smallest cycle or \emph{even cover} (a subhypergraph with all degrees even) and size --- the number of hyperedges in a hypergraph. For graphs (i.e., $2$-uniform hypergraphs), a bound tight up to the leading constant was proven in a classical work of Alon, Hoory and Linial~\cite{AHL02}. For hypergraphs of uniformity $k>2$, an appropriate generalization was conjectured by Feige~\cite{Fei08}. The conjecture was settled up to an additional $\log^{4k+1} n$ factor in the size in a recent work of Guruswami, Kothari and Manohar~\cite{GKM21}. Their argument relies on a connection between the existence of short even covers and the spectrum of a certain randomly signed \emph{Kikuchi} matrix. Their analysis, especially for the case of odd $k$, is significantly complicated. 

In this work, we present a substantially simpler and shorter proof of the hypergraph Moore bound. Our key idea is the use of a new \emph{reweighted} Kikuchi matrix and an \emph{edge deletion} step that allows us to drop several involved steps in \cite{GKM21}'s analysis such as combinatorial bucketing of rows of the Kikuchi matrix and the use of the Schudy--Sviridenko polynomial concentration. Our simpler proof also obtains tighter parameters: in particular, the argument gives a new proof of the classical Moore bound of~\cite{AHL02} with no loss (the proof in \cite{GKM21} loses a $\log^3 n$ factor), and loses only a single logarithmic factor for all $k>2$-uniform hypergraphs.


As in~\cite{GKM21}, our ideas naturally extend to yield a simpler proof of the full trade-off for strongly refuting smoothed instances of constraint satisfaction problems with similarly improved parameters.

\end{abstract}

\thispagestyle{empty}
\setcounter{page}{0}
\newpage

\section{Introduction}

What is the maximum girth of a graph on $n$ vertices and average degree $d$? For $d$-regular graphs, a simple ``ball growing'' argument shows that the graph must have a cycle of length at most $2\log_{d-1} n+2$. This threshold is called the \emph{Moore bound}~\cite{enwiki:1096680615} (see Page 180 of~\cite{MR1271140}) and graphs achieving it are called Moore graphs. In a classical paper that resolved a question of \Bollobas~\cite{MR506522}, Alon, Hoory and Linial~\cite{AHL02} proved that the same upper bound holds even for irregular graphs. Later on, Hoory~\cite{hoory2002size} obtained a better bound for bipartite graphs and Babu and Radhakrishnan~\cite{BR14} found an elegant proof based on the entropy of random walks. 

\medskip
\noindent {\bf Girth-density trade-offs for hypergraphs.}
This work is about a natural and well-studied generalization of the Moore bound to $k>2$-uniform hypergraphs. A cycle\footnote{There are several well-studied combinatorial notions of \href{https://en.wikipedia.org/wiki/Hypergraph\#Cycles}{cycles} in contrast to the more linear algebraic notion of even covers.} in a hypergraph, more descriptively called an \emph{even cover}, is a collection of hyperedges such that every vertex participates in an even number of them. The girth of a hypergraph is the smallest size of an even cover in it. When specialized to graphs, an even cover is simply a union of cycles and thus, this formulation naturally generalizes the standard notion of girth in graphs. 

Analogously to the Moore bound, understanding the maximum number of hyperedges that one can pack in a hypergraph while avoiding an even cover of a given length is a basic \emph{hypergraph \Turan} problem. Hypergraph \Turan problems are typically significantly more difficult than their counterparts in graphs. Indeed, even the original hypergraph \Turan conjecture from the 1940s that studies an appropriate analog of triangle free hypergraphs is still open. We direct the reader to the recent survey of Keevash~\cite{Kee11} for an overview of hypergraph Turán theory. 

\medskip
\noindent {\bf Applications of girth-density trade-offs.} Like the graph Moore bound, girth-density trade-offs for hypergraphs have foundational connections to several research directions in theoretical computer science. One source of such applications is the observation that for a collection of linear equations in $\F_2$ on $n$ variables, if we associate each equation to the set indicated by its coefficient vector, then the girth of the resulting hypergraph on $[n]$ is the same as the size of the smallest linearly dependent subset of equations. As a consequence, rate vs distance trade-offs for \emph{low density parity check} (LDPC) codes are equivalent to the girth vs size trade-offs for $k$-uniform hypergraphs with hyperedges corresponding to the columns of the parity check matrix. As a result, there is an extensive line of work that studies the girth-density trade-offs for hypergraphs (see e.g.~\cite{BKHL99,BMS08,AF09}).

Naor and Verstraëte~\cite{NV08} started a systematic study of hypergraph girth density trade-offs. They were explicitly motivated by mapping the rate-distance trade-offs for LDPC codes and computing product representations of square integers arises as a step in sub-exponential time algorithms for integer factoring. In particular, they showed that every $k$-uniform hypergraph on $n$ vertices and $O(n^{k/2} \log n)$ hyperedges must have an even cover of size $O(\log n)$. Improving the bounds of~\cite{NV08} for $k=3$, Feige~\cite{Fei08} proved that every 3-uniform hypergraph on $n$ vertices and $O(n^{3/2}) \log \log n$ hyperedges has an even cover of length $O(\log n)$. Feige's motivation was a connection, via the connection to linear equations modulo 2 discussed above, to refuting random 3SAT formulas and generalizations. In particular, by exploiting this improved bound, Feige derived a weak refutation algorithm for smoothed 3SAT formulas with $O(n^{1.5} \log \log n)$ constraints.

\medskip
\noindent \textbf{Hypergraph Moore bound.}
Feige's result leaves open the uncharted territory of hypergraph sizes between $m \sim n$ and $m\sim n^{k/2}$ --- a polynomially large multiplicative interval when $k>2$. The work of Feige, Kim and Ofek~\cite{FKO06} found an intriguing connection between the girth bounds in this interesting regime and the foundational average-case problem of refuting random 3SAT formulas~\cite{Fei02}. They observed that \emph{random} hypergraphs with $m \gtrsim n^{1.4}$ hyperedges\footnote{Throughout this work, we will use the notation $f \gtrsim g$ to stand for ``there exists a constant $C>0$ such that $f \geq Cg$''.} must have an even cover of length $O(n^{0.2})$ and used a tour de force argument based on the second moment method to establish that at the same density, random hypergraphs should contain $n^{1.4}$ different almost disjoint even covers of size $n^{0.2}$. As a consequence, they obtained their celebrated result on the existence of polynomial size witnesses of unsatisfiability for random 3SAT formulas with $O(n^{1.4})$ constraints --- a threshold that is $n^{0.1}$ factor smaller than the $m \gtrsim n^{1.5}$ bound for the best known efficient refutation algorithms. Motivated both by whether FKO witnesses could be efficiently constructed (and potentially refute a strong form of Feige's Random 3SAT hypothesis~\cite{Fei02}) and investigating whether such certificates exist in semirandom and smoothed 3SAT formulas, Feige~\cite{Fei08} conjectured the following \emph{hypergraph Moore bound}. 

\begin{conjecture}[Hypergraph Moore Bound (Feige's conjecture), Conjecture 1.2 of~\cite{Fei08}] \label{conj:feige-conjecture}
    For every $k \in \N$ and $1 \leq r \leq n$, every hypergraph with $n$ vertices and $m \gtrsim n (\frac{n}{r})^{\frac{k}{2}-1}$ hyperedges has an even cover of size $O(r\log n)$.
\end{conjecture}

In addition to a complete rate-distance profile for LDPC codes, Feige's conjecture implies (see Section 9 in~\cite{GKM21} for an exposition) a significantly simpler and 2nd-moment-method-free proof of the existence of the FKO~\cite{FKO06} refutation witnesses below the spectral threshold for random 3SAT (and other CSPs) that also generalizes to semirandom and smoothed instances\footnote{A smoothed Boolean CSP instance is obtained by starting from a worst-case instance and perturbing the literal patterns by independently flipping each with some small constant probability (with probability 1/2 in the special case of the semirandom model). In particular, in contrast to random CSPs where the variables in every clause are generated uniformly at random, smoothed and semirandom CSP instances have a worst-case clause structure.}. 

Feige's conjecture was recently settled by Guruswami, Kothari and Manohar~\cite{GKM21} up to an additional $\log^{4k+1} n$ multiplicative factor in the density $m$. Their proof goes via a new connection between the existence of small even covers in $k$-uniform hypergraphs and sub-exponential size spectral refutations of semirandom $k$-XOR formulas via a certain \emph{Kikuchi} matrix.

While \cite{GKM21} begins with an elegant and simple observation, their technical analysis especially for  odd $k$ (the ``hard'' case in all algorithms and certificates for refutation) is quite complicated and involves manipulating the Kikuchi matrix via ``row bucketing'' and ``row pruning'' in various steps and invoking the Schudy--Sviridenko concentration inequality~\cite{SS12} (that extends the breakthrough work of Kim and Vu~\cite{KV00}) for polynomials with combinatorial structure in the monomials.
As a consequence, even for the simplest case of $k=2$ (i.e., recovering the classical Moore bound), their proof incurs an additional $\log^3 n$ factor. 

\subsection{Our results}
The main result of this work is a simple and short proof of the hypergraph Moore bound that is \emph{almost} tight up to a single logarithmic factor. 

\begin{theorem} \label{thm:main-thm}
    For every $k \in \N$ and $1 \leq r \leq n$, every hypergraph on $n$ vertices and $m \gtrsim n \log n \cdot (\frac{n}{r})^{\frac{k}{2}-1}$ hyperedges has an even cover of size $O(r\log n)$.
\end{theorem}
In \pref{sec:weak-moore-bound} and \pref{app:exact-moore-bound}, as evidence of the power of our proof strategy, we obtain yet another proof of the classical Moore bound~\cite{AHL02} with the same leading constant. Independently of our work, David Munh\'a Correia and Benny Sudakov~\cite{CS22} informed us that they have found a simple, combinatorial argument for analyzing the Kikuchi matrix to prove a hypergraph Moore bound for even arity $k$ that also loses only a single logarithmic factor. 

Our techniques extend to give a simple and tighter proof of a sub-exponential time strong refutation algorithm for semirandom $k$-XOR formulas when the number of constraints is below the ``spectral threshold'' $n^{k/2}$, which is spelled out in \pref{sec:refutation}. Via the standard XOR trick (see, for example, \cite{AOW15}), this recovers a tighter trade-off for refuting smoothed Boolean constraint satisfaction problems as in~\cite{GKM21}. Prior to our work, a bound tight up to $\log n$ factors was not known even for the (significantly) easier setting of \emph{fully random} $k$-XOR refutation for odd $k$ (the argument of~\cite{WAM19} obtains such a result for even $k$) where the best known bound due to~\cite{RRS17} loses a $\log^{2k} n$ factor.

\begin{theorem}[Informal]
    Fix $k\in \N$ and $r \leq n$, there is an $n^{O(r)}$-time algorithm such that given a semirandom $k$-XOR instance $\psi$ with $n$ variables and $m \gtrsim n\log n \cdot (\frac{n}{r})^{\frac{k}{2}-1}$ constraints, it certifies that $\psi$ is not $(1/2+0.01)$-satisfiable.
\end{theorem}

We believe that the last remaining logarithmic factor in the theorems above is also unnecessary. However, removing it seems related to certain technical difficulties that arise in beating the logarithmic factor incurred in spectral norm bounds for the matrix Rademacher series~\cite{Tropp}. In particular, the tightest known proof for the closely related problem of refuting fully random $k$-XOR formulas below the spectral threshold also loses a $\log n$ factor for even $k$~\cite{WAM19} and a $\log^{2k}n$-factor for odd $k$~\cite{RRS17}. For some easier settings such as refutation in the polynomial time regime~\cite{dT22} and understanding the SDP value of random NAE-3SAT and generalizations~\cite{FM17,DMO19,MOP20}, recent works manage to circumvent this difficulty by application of powerful tools such as the Ihara--Bass formula and the largest eigenvalue of non-backtracking walk matrices. Our proof suggests a natural and more elementary route to removing this final logarithmic factor but requires resolving the count of certain ``walks" that arise in our analysis.

\paragraph{Key ideas.} The abstract strategy employed by \cite{GKM21} is to construct a so called \emph{Kikuchi} matrix $A_{\calH}$ (first introduced in the work of~\cite{WAM19} on Gaussian tensor PCA) associated with our hypergraph $\calH$ where:
\begin{enumerate}[(i)]
    \item $\1^{\top}A_{\calH}\1$ is a surrogate for the number of hyperedges in $\calH$.
    \item The lack of short even covers in $\calH$ can be turned into a certificate that $\1^{\top}A_{\calH}\1$ is small, which then translates to a bound on $|\calH|$.
\end{enumerate}

The certificate used by \cite{GKM21} is $\Norm{A_{\calH}}_{\infty\to1}$, which they control by bucketing the rows by weight, bounding the spectral norm of each submatrix, and stitching these norms together.

Our key insight is in the style of certificate we provide --- we give a matrix $Q$ such that $Q\psdge A_{\calH}$.
Such a certificate implies a bound of $\tr(Q)$ on our surrogate for $|\calH|$.
The inequality $Q\psdge A_{\calH}$ is equivalent to proving $\Norm{Q^{-1/2}A_{\calH}Q^{-1/2}}_2 \le 1$, which can be done via the trace moment method with relative ease.
Our reweighting strategy is akin to constructing a ``diagonal weighted'' \emph{dual solution} for certifying upper bounds on the value of the basic SDP relaxation for quadratic optimization problems on the hypercube such as Max-Cut and the Grothendieck problems.
Our \emph{reweighting} strategy simplifies the analysis, removes the need for the ``row bucketing'' step in \cite{GKM21}, and lets us obtain a sharper result.

Our proof for odd $k$ requires combining our reweighted Kikuchi matrix with a new ``edge deletion'' operation that controls the ``heavy rows'' in the Kikuchi matrix. At a high level, our strategy involves deleting an appropriately chosen set of entries of the Kikuchi matrix in comparison to the ``row pruning" strategy of~\cite{GKM21} which involves deleting entire rows (vertices in the Kikuchi graph). This seemingly technical change leads to a great deal of simplification and in particular allows replacing the use of the Schudy--Sviridenko inequality~\cite{SS12} and the carefully introduced logarithmic factors in the hypergraph regularity decomposition in~\cite{GKM21}.

\paragraph{Organization.}
The rest of this paper is organized as follows.
In \pref{sec:warmup}, we give a complete (and sharper) proof of the hypergraph Moore bound for the even arity case, starting with a proof of a weak version (that loses additional constant factors) of the classical Moore bound using our ideas. We will include detailed commentary for the sake of exposition and a short overview of the additional ideas (including our new edge deletion trick) to handle the odd arity case. In \pref{sec:even-cover-odd}, we will give a proof of the hypergraph Moore bound for odd arity. Finally, in \pref{sec:refutation}, we will extend our techniques to obtain strong refutation algorithms for semirandom and smoothed Boolean CSPs.
\newcommand{\cH}{\mathcal{H}}
\section{Warm-up: hypergraph Moore bound in the even arity case}
\label{sec:warmup}

In this section, we will give a proof of the Moore bound for hypergraphs of even arity with the goal of providing an exposition of our main ideas. 
As an illustration of the power of our reweighting idea, in \pref{sec:weak-moore-bound} we will give a simple  proof of the classical Moore bound~\cite{AHL02} that is tight up to an absolute constant factor (as opposed to the $\log^3 n$ loss incurred by the strategy of~\cite{GKM21}).\footnote{In \pref{app:exact-moore-bound}, we present a proof that uses one additional tool to recover the classical Moore bound for irregular graphs with the same leading constant.}
In \pref{sec:even-cover-even}, we will generalize the reweighting idea to prove hypergraph Moore bound for all even arities.
Finally in \pref{sec:odd-arity-overview}, we will discuss the key new idea of \emph{edge deletions} that is crucial for our simpler and tighter proof for the case of odd $k$.

\subsection{Weak Moore bound for graphs}
\label{sec:weak-moore-bound}

In this section, we prove a weak Moore bound for graphs to illustrate our reweighting strategy in a simple setting. The resulting bound is weak in the sense that it incurs a constant factor loss when compared to~\cite{AHL02}. In \pref{app:exact-moore-bound}, we implement this strategy (in a way that is less generalizable to hypergraphs) to recover the tight $2\log_{d-1}n$ bound. 

We note that \cite{GKM21} also proved a weaker Moore bound (Proposition~2.3 of \cite{GKM21}) to illustrate their ``row bucketing" strategy that partitions the vertices into $O(\log n)$ buckets, each of which has vertices with degrees within a multiplicative constant factor of each other. This strategy splits the adjacency matrix $A$ into $O(\log^2 n)$ pieces and ends up requiring an average degree $d \gtrsim \log^3 n$ in order to contain a cycle of length $O(\log n)$. 

This simple exercise will show how our reweighting handles different degrees automatically, avoiding the lossy row bucketing step completely.
\begin{proposition}[Weak Moore bound for irregular graphs]
\label{prop:weak-moore-bound}
    Every graph with $n$ vertices and average degree $d > 16 $ has a cycle of length at most $2\ceil{\log_{(d/16)} n}$.
\end{proposition}

The core of the proof of \pref{prop:weak-moore-bound} is the following spectral norm bound on the reweighted adjacency matrix.

\begin{claim} \label{claim:weak-moore-bound-norm}
    Let $G$ be a graph with $n$ vertices and average degree $d > 1$ that has no cycle of length $\leq \ell$ for some even $\ell\in\N$.
    Let $A$ be the $\{0,1\}$ adjacency matrix of $G$, and let $\Gamma = D+d\Id$ be the diagonal matrix such that $D_{uu} = d_u$ where $d_u$ is the degree of vertex $u$. Then,
    $\Norm{\Gamma^{-1/2} A \Gamma^{-1/2}}_2 < \frac{2n^{1/\ell}}{\sqrt{d}}$.
\end{claim}

We now complete the proof of \pref{prop:weak-moore-bound}.
\begin{proof}[Proof of \pref{prop:weak-moore-bound} by \pref{claim:weak-moore-bound-norm}]
    Suppose $G$ has no cycle of length $\leq \ell$, then \pref{claim:weak-moore-bound-norm} implies that $A \prec \frac{2n^{1/\ell}}{\sqrt{d}} \Gamma$.
    Then, the quadratic form $\1^\top A \1 < \frac{2n^{1/\ell}}{\sqrt{d}} \tr(\Gamma)$ since $\1^\top \Gamma \1 = \tr(\Gamma)$. By definition, $\1^\top A \1 = nd$ and $\tr(\Gamma) = \sum_{u=1}^n (d_u + d) = 2nd$. Thus, $n^{1/\ell} > \sqrt{d}/4$, and taking logs, we get
    \begin{equation*}
        \frac{1}{\ell}\log n > \frac{1}{2} \log(d/16)
        \Rightarrow \frac{\ell}{2} < \log_{d/16} n \mper
    \end{equation*}
    $\ell$ is even, so we have $\ell < 2\ceil{\log_{d/16} n}$.
    Thus, by the contrapositive, $G$ must contain a cycle of length $2\ceil{\log_{d/16} n}$. This completes the proof.
\end{proof}

We now prove \pref{claim:weak-moore-bound-norm} using the well-known trace moment method, which reduces to counting weighted closed walks in the graph.
In the analysis, we will see exactly how the choice of the reweighting matrix $\Gamma$ accounts for different vertex degrees.

\begin{proof}[Proof of \pref{claim:weak-moore-bound-norm}]
    Let $\wt{A} = \Gamma^{-1/2} A \Gamma^{-1/2}$.
    For even $\ell\in\N$, the trace moment method states that $\norm{\wt{A}}_2^\ell \leq \tr(\wt{A}^\ell) = \tr((\Gamma^{-1}A)^{\ell})$, which is a summation of all (weighted) closed walks of length $\ell$ in $G$.
    Since there is no cycle of length $\leq \ell$, the only closed walks are the ones that \emph{backtrack} to the original vertex, meaning that there can be at most $\ell/2$ ``new'' edges and at least $\ell/2$ ``old'' edges in the walk.
    We encode each closed walk $u_1 \to u_2 \to \cdots \to u_\ell \to u_1$ as follows,
    \begin{itemize}
        \item Choose a starting vertex $u_1 \in [n]$.
        \item One bit $b_i\in \{0,1\}$ at each step $i$ to encode whether this step uses a new edge or an old one.
        \begin{itemize}
            \item If $b_i=0$ (new edge), select one of $u_i$'s neighbors as $u_{i+1}$.
            \item If $b_i=1$ (old edge), we must backtrack to the previous vertex $u_{i-1}$.
        \end{itemize}
    \end{itemize}
    For $b\in\{0,1\}$ and $u\in [n]$, let $N_b(u) \subseteq [n]$ be the possible next steps in the walk from $u$. Then, simply expanding $\tr((\Gamma^{-1} A)^\ell)$, we get
    \begin{equation*}
        \tr((\Gamma^{-1} A)^\ell) = \sum_{b\in\zo^\ell} \sum_{u_1 \in [n]} \sum_{u_2 \in N_{b_1}(u_1)} \Gamma_{u_1u_1}^{-1} \sum_{u_3 \in N_{b_2}(u_2)} \Gamma_{u_2u_2}^{-1} \cdots \sum_{u_{\ell+1} \in N_{b_{\ell}}(u_{\ell})} \Gamma_{u_{\ell}u_{\ell}}^{-1}
        \cdot \1(u_{\ell+1} = u_1) \mper
    \end{equation*}
    As we can see, each step $u_i\to u_{i+1}$ gets a factor $\Gamma_{u_iu_i}^{-1} = \frac{1}{d_{u_i}+d}$.
    We can now bound the above by observing that if $b_i=0$ (new edge), then $|N_0(u_i)| \leq d_{u_i}$ and
    \begin{equation*}
        \sum_{u_{i+1}\in N_0(u_i)} \Gamma_{u_i u_i}^{-1} \leq \frac{d_{u_i}}{d_{u_i}+d} < 1 \mcom
    \end{equation*}
    and if $b_i = 0$ (old edge), then $|N_1(u_i)| = 1$ (the previous step) and
    \begin{equation*}
        \sum_{u_{i+1}\in N_1(u_i)} \Gamma_{u_i u_i}^{-1} \leq \frac{1}{d_{u_i}+d} < \frac{1}{d} \mper
    \end{equation*}
    Finally, considering $b\in\zo^{\ell}$, $u_1\in [n]$, and there are at least $\ell/2$ old edges, we have
    \begin{equation*}
        \tr((\Gamma^{-1} A)^\ell) < 2^\ell n \Paren{\frac{1}{d}}^{\ell/2} \mcom
    \end{equation*}
    and taking the $\ell$-th root completes the proof.
\end{proof}

\subsection{The case of even arity hypergraphs}
\label{sec:even-cover-even}

In this section, we prove the existence of small even covers in even arity hypergraphs.

\begin{theorem}[\pref{thm:main-thm}, even $k$]
\label{thm:even-arity}
    For even $k\in \N$ and any $r \in \N$ with $k \leq r \leq n/8$, any $k$-uniform hypergraph $\calH$ with $n$ vertices and $m \geq 128 n\log n \cdot (\frac{n}{r})^{k/2-1}$ hyperedges has an even cover of size at most $\ceil{r \log_2 n}+1$.
\end{theorem}

The proof is simple and almost identical to the proof of the weak Moore bound (\pref{prop:weak-moore-bound}) but with $A$ being the adjacency matrix of the \emph{Kikuchi graph} which we define below.

\begin{definition}[Kikuchi graph]
\label{def:even-kikuchi-graph}
    Let $\cH$ be a $k$-uniform hypergraph on vertex set $[n]$ for even $k$. For an integer parameter $r$, define the \emph{Kikuchi graph} $K_{r}$ associated to $\cH$ is a graph on vertex set $\binom{[n]}{r}$ such that a pair of vertices $S, T \in \binom{[n]}{r}$ have an edge between them if the symmetric difference $S \oplus T \in \cH$. For such an edge, we write $S \xleftrightarrow{C} T$ and think of the edge as ``colored'' by $C \in \cH$ where $C = S \oplus T$. 
    We call the adjacency matrix $A$ of $K_r$ the \emph{Kikuchi matrix.}
\end{definition}

The key insight of~\cite{GKM21} (and also our starting point) is relating even covers in $\cH$ to cycles in the associated Kikuchi graph. For sets $R_1, R_2, \ldots, R_\ell \subseteq [n]$ let $\oplus_{i \leq \ell} R_i$ denote the set of elements of $[n]$ that appear in an odd number of $R_i$s (i.e., the sum modulo $2$ of the indicator vectors of $R_i$s).
  
\begin{observation}[Closed walks in the Kikuchi graph]
\label{obs:closed-walks}
Let $\cH$ be a $k$-uniform hypergraph on $[n]$ for even $k$ and let $S_1 \rightarrow S_2 \rightarrow \cdots S_\ell \rightarrow S_1$ be a closed walk on vertices in $K_r$ such that for every $i \leq \ell$, $S_i \xleftrightarrow{C_i} S_{i+1}$ for $C_1, C_2, \ldots, C_{\ell} \in \cH$ (denoting $S_{\ell+1} = S_1$). Then, $\oplus_{i \leq \ell} C_i = 0$. Further, if $\cH$ has no even cover of length $\ell$, then every hyperedge in $\cH$ appears an even number of times in the multiset $\{C_1, C_2, \ldots, C_\ell\}$. We will call such walks in $K_{r}$ \emph{trivial}.
\end{observation}

\begin{proof}
Note that $S_{i} \oplus S_{i+1} = C_i$ for every $i \leq \ell$. If we add both sides of all $\ell$ such equalities then each $S_i$ occurs in exactly two of the equations so the LHS must be $0$. Thus, $\oplus_{i \leq \ell} C_i = 0$.

Next, we repeatedly remove hyperedges that occur an even number of times in the multiset $\{C_1, C_2, \ldots, C_\ell\}$ to obtain a collection of $\ell' \leq \ell$ distinct hyperedges of $\cH$. The sum (modulo 2) of the remaining hyperedge should still be $0$ as we removed hyperedges in pairs. The resulting $\ell'$ must be $0$ as otherwise the remaining hyperedges form an even cover of length $\ell' \leq \ell$. 
\end{proof}

Consider a hypergraph $\calH$ with $n$ vertices and $m$ hyperedges, and its associated Kikuchi graph $(V,E)$ with parameter $r$.
Each $C\in \calH$ introduces $\frac{1}{2}\binom{k}{k/2} \binom{n-k}{r-k/2}$ edges in the Kikuchi graph (select $k/2$ vertices from $C$ and select $r-k/2$ vertices from $[n]\setminus C$ to complete $S$), thus the total edges $|E| = \frac{1}{2}\binom{k}{k/2} \binom{n-k}{r-k/2} \cdot m$.
Let $d_S$ be the degree of $S\in V$, and let $d$ denote the average degree, then simple calculations show that
\begin{equation} \label{eqn:even-average-degree}
    d = \frac{\binom{k}{k/2} \binom{n-k}{r-k/2} m}{\binom{n}{r}}
    \geq \Paren{\frac{r}{n}}^{k/2} m \cdot \binom{k}{k/2} \Paren{1-\frac{2r}{n}}^{k/2} \Paren{1 - \frac{k}{2r}}^{k/2}
    \geq \frac{1}{2}\Paren{\frac{r}{n}}^{k/2} m
\end{equation}
when $k \leq r \leq n/8$.

We will follow the reweighting strategy with $\Gamma = D + d \Id$ to bound the spectral norm of the reweighted Kikuchi matrix.
The following lemma is analogous to \pref{claim:weak-moore-bound-norm}.

\begin{lemma} \label{lem:even-kikuchi-norm-bound}
    Let $k,r,n\in \N$ such that $k\leq r \leq n$, and let $\ell \in \N$ be even.
    Let $A$ be the Kikuchi matrix with parameter $r$ of a $k$-uniform hypergraph $\calH$ on $n$ vertices, and let $\Gamma = D + d\Id$ where $D$ is the degree matrix and $d$ is the average degree of the Kikuchi graph.
    Suppose there is no even cover of size at most $\ell$ in $\calH$, then
    \begin{equation*}
        \Norm{\Gamma^{-1/2}A \Gamma^{-1/2}}_2 < 2n^{r/\ell}\sqrt{\frac{\ell}{d}} \mper
    \end{equation*}
\end{lemma}

We can immediately complete the proof of \pref{thm:even-arity}.

\begin{proof}[Proof of \pref{thm:even-arity} by \pref{lem:even-kikuchi-norm-bound}]
    Suppose that there is no even cover of size $\leq \ell \coloneqq \ceil{r\log_2 n}$ (assume this is even, otherwise add 1).
    Then, $n^{r/\ell}\leq 2$ and \pref{lem:even-kikuchi-norm-bound} states that the Kikuchi graph $(V,E)$ satisfies $A \prec 4\sqrt{\ell/d} \cdot \Gamma$ where $\Gamma = D + d\Id$.
    Then,
    \begin{equation*}
        \1^\top A \1 < 4\sqrt{\frac{\ell}{d}} \cdot \tr(\Gamma)
        = 4\sqrt{\frac{\ell}{d}} \cdot \sum_{S\in V} (d_S + d)
        = 8\sqrt{\frac{\ell}{d}} \cdot |V| d \mper
    \end{equation*}
    On the other hand, $\1^\top A \1 = 2|E| = |V| d$.
    Thus, we have $d < 64\ell$.
    By \pref{eqn:even-average-degree} we have $d \geq \frac{1}{2}(\frac{r}{n})^{k/2} m$ when $k \leq r \leq n/8$.
    Thus, if there is no even cover of size $\leq \ell$, then $m < 128n\log n \cdot (\frac{n}{r})^{k/2-1}$, completing the proof.
\end{proof}

Now, we prove \pref{lem:even-kikuchi-norm-bound} by counting weighted closed walks in the Kikuchi graph, essentially the same way we prove \pref{claim:weak-moore-bound-norm}.

\begin{proof}[Proof of \pref{lem:even-kikuchi-norm-bound}]
    Let $\wt{A} = \Gamma^{-1/2}A\Gamma^{-1/2}$.
    We use the trace power method:
    \begin{equation*}
        \|\wt{A}\|_2^\ell \leq \tr(\wt{A}^\ell) = \tr( (\Gamma^{-1} A)^\ell) \mper
    \end{equation*}
    We upper bound $\tr((\Gamma^{-1} A)^\ell)$ by counting (weighted) closed walks of length $\ell$ in the Kikuchi graph.
    Note that each edge $(S,T)$ of the Kikuchi graph corresponds to a hyperedge $S \oplus T \in \calH$.
    Since there is no even covers of size at most $\ell$, any closed walk must contain an even number of each hyperedge in $\calH$.

    We can encode a closed walk $S_1 \rightarrow S_2 \rightarrow \cdots \rightarrow S_{\ell} \rightarrow S_1$ as follows:
    \begin{itemize}
        \item Choose a starting vertex $S_1 \in V$.
        \item One bit $b_i \in \zo$ at each step $i$ to encode whether this step uses a new hyperedge or an old one.
        \begin{itemize}
            \item If $b_i = 0$ (new hyperedge), select one of $S_i$'s neighbors as $S_{i+1}$.
            \item If $b_i = 1$ (old hyperedge), select an old hyperedge $C$ from the previous steps, and set $S_{i+1} = S_i \oplus C$.
        \end{itemize}
    \end{itemize}
    Note that there are at most $\ell/2$ new hyperedges and at least $\ell/2$ old hyperedges since each hyperedge must occur an even number of times.
    For $b \in \zo$ and $S\in V$, let $N_b(S) \subseteq V$ be the possible next steps in the walk from $S$ (according to $b$).
    Each step $S_i \rightarrow S_{i+1}$ gets a factor $(\Gamma^{-1}A)_{S_i, S_{i+1}} = \Gamma_{S_i,S_i}^{-1} = \frac{1}{d_{S_i}+d}$.
    Thus,
    \begin{equation*}
        \tr((\Gamma^{-1} A)^\ell) = \sum_{b\in\zo^\ell} \sum_{S_1 \in V} \sum_{S_2 \in N_{b_1}(S_1)} \frac{1}{d_{S_1} + d} \sum_{S_3 \in N_{b_2}(S_2)} \frac{1}{d_{S_2} + d} \cdots \sum_{S_{\ell+1} \in N_{b_{\ell}}(S_{\ell})} \frac{\1(S_{\ell+1} = S_1)}{d_{S_{\ell}} + d} \mper
    \end{equation*}

    We can upper bound the above as follows.
    If $b = 0$, then $|N_0(S_i)| \leq d_{S_i}$ and $\sum_{S_{i+1} \in N_0(S_i)} \Gamma_{S_iS_i}^{-1} \leq \frac{d_{S_i}}{d_{S_i} + d} < 1$.
    If $b = 1$, then $|N_1(S_i)| \leq \ell$ as there are only $\ell$ options to choose one of the previous steps, and $\sum_{S_{i+1} \in N_1(S_i)} \Gamma_{S_iS_i}^{-1} \leq \frac{\ell}{d_{S_i} + d} < \frac{\ell}{d}$.
    Furthermore, we can assume that $\ell \leq d$, otherwise we can simply treat all steps as new hyperedges.

    Finally, $b\in\{0,1\}^\ell$, there are $|V| = \binom{n}{r}$ choices for the starting vertex $S_1$, and there are at least $\ell/2$ old hyperedges.
    Thus, we have
    \begin{equation*}
        \tr((\Gamma^{-1} A)^\ell) < 2^\ell \binom{n}{r} \Paren{\frac{\ell}{d}}^{\ell/2}
        \leq 2^\ell n^r \Paren{\frac{\ell}{d}}^{\ell/2} \mper
    \end{equation*}
    Taking the $\ell$-th root completes the proof.
\end{proof}
\label{sec:overview-hypergraph-mb}

\paragraph{Summary.} As the above short and simple arguments illustrate, reweighted Kikuchi matrix appears to be a clean and simple way to handle irregularities in the degree of graphs (Kikuchi or otherwise) in spectral double counting. For $k>2$, we suspect that the extra $\log n$ factor incurred in our analysis can likely be removed by a better counting of the weighted closed walks.

\subsection{Overview of the odd arity case}
\label{sec:odd-arity-overview}

As in many previous works on refuting constraint satisfaction problems, the odd arity case requires significantly more work. Indeed, even the definition of the Kikuchi graph (\pref{def:even-kikuchi-graph}) only makes sense when $k$ is even.  We present the proof of the odd arity case in \pref{sec:even-cover-odd}, and here we outline some of our key ideas.

\paragraph{Bipartite hypergraph.}
The main insight is to transform the hypergraph $\calH$ to a ``bipartite'' hypergraph (this abstraction is closely related to the Cauchy-Schwarz trick in the context of odd-arity CSP refutation).
First, we partition the hyperedges of $\calH$ into $\calH_1,\dots,\calH_p$ such that for each $\calH_i$, all hyperedges in $\calH_i$ contains a ``center'' vertex $u_i\in [n]$.
We denote $\wt{\calH}_i$ to be $\{C\setminus \{u_i\}: C\in \calH_i\}$, i.e.\ removing the center vertex, and denote $\wt{C} \coloneqq C\setminus \{u_i\}$.
Then, we construct a $2(k-1)$-uniform hypergraph as follows: for each $i\in[p]$ and each distinct pair $C,C'\in \calH_i$, we add a hyperedge $\wt{C}\oplus \wt{C}'$ (let's assume $\wt{C},\wt{C}'$ are disjoint for now).

Let's make some quick calculations.
Suppose that $\calH$ has $m$ hyperedges, and suppose that there are roughly $p\approx n$ partitions and each partition size is $\approx m/n$.
Then, the new $2(k-1)$-uniform hypergraphs will contain roughly $n\cdot (m/n)^2 = \frac{m^2}{n}$ hyperedges.
Now, since $2(k-1)$ is even, we can apply our bound for the even case: there is an even cover of size $r\log n$ when $\frac{m^2}{n} \geq \wt{O}(n) (\frac{n}{r})^{(k-1)-1}$, meaning $m \geq \wt{O}(n) (\frac{n}{r})^{\frac{k}{2}-1}$, the correct bound!

The issue is that this new hypergraph has small even covers for trivial reasons: any 3 pairs $(C_1, C_2)$, $(C_2,C_3)$ and $(C_3,C_1)$ from $\calH_i$ form an even cover of size 3.
Nevertheless, we can proceed to analyze the Kikuchi matrix of the new hypergraph (\pref{def:odd-kikuchi-graph}), assuming that there is no small even cover in the \emph{original} hypergraph.
Note that now an edge $(S,T)$ is associated with 2 hyperedges $C,C'$ from the same $\calH_i$, which we denote as $S \xleftrightarrow{C,C'} T$.
Assuming that there is no even cover of size $\leq 2\ell$, we bound the number of ``trivial'' closed walks where each hyperedge is used an even number of times.

\paragraph{Encoding a closed walk.}
The standard technique of bounding counts of closed walks in the trace moment method is to give a small \emph{encoding} of a walk.
In our case, in a length-$\ell$ closed walk of the Kikuchi graph, each step is associated with two hyperedges, and we have two types of steps:
\begin{enumerate}
    \item a step using 2 \emph{new} hyperedges, and
    \item a step using at least 1 \emph{old} hyperedge.
\end{enumerate}
The first type is bounded exactly the same way as the even arity case by our weight matrix $\Gamma$, the trouble is the second type: while we can easily encode one edge in the step, we need too many bits to encode the other edge.

\paragraph{Deleting bad edges of the Kikuchi graph.}
The main insight is that in the end, we only care about bounding $\1^\top A \1$.
Again, let $d$ be the average degree of the Kikuchi graph, $A\in \R^{N\times N}$ be the Kikuchi matrix, and $\Gamma = D + d\Id$ be our diagonal weight matrix.
If we delete (say) half of the edges of the Kikuchi graph such that we have $\|\Gamma^{-1/2}A' \Gamma^{-1/2}\|_2 \leq \lambda(d)$, where $A'$ is the modified Kikuchi matrix, for some ``good enough'' $\lambda(d)$, then we will have $\frac{Nd}{2} \leq \1^\top A' \1 \leq \lambda \tr(\Gamma) = \lambda \cdot 2Nd$, essentially only losing a constant factor in the density.

We define an appropriate edge deletion process, prove that the fraction of edges removed is small (\pref{claim:few-deletions}), and show that the resulting subgraph has combinatorial properties that let us encode steps of the second type efficiently (\pref{lem:good-subgraph-bound}).


\paragraph{Improving the row pruning step of \cite{GKM21}.}
The analysis of~\cite{GKM21} also requires reducing the Kikuchi graph to obtain certain combinatorial properties.
However, instead of deleting ``bad'' edges, they delete ``bad'' vertices, which they defined as vertices that are bad for \emph{some} $i$ in the bipartite hypergraph (they call this row pruning as each row of the Kikuchi matrix corresponds to a vertex).
Crucially, doing so requires a union bound over $i$, hence they need a strong bound on fraction of bad vertices for each $i$.
Furthermore, they proved their bound using tail inequalities for low-degree polynomials by Schudy and Sviridenko~\cite{SS12}, which is a powerful black-box concentration inequality but loses log factors and requires involved analysis.
All this combined with the row bucketing step introduces several log factors.

\paragraph{Hyperedges with large intersections.}
It turns out that the fraction of bad edges highly depends on \emph{large intersections} of hyperedges in $\calH$.
To bound the fraction of edges deleted, we require our hypergraph to be somewhat ``regular'' -- that is, no small subset appears in more than an appropriately chosen threshold of hyperedges in $\cH$. To this end, we invoke the hypergraph regularity decomposition of~\cite{GKM21} (with more transparently chosen thresholds that do not involve carefully chosen logarithmic factors) to decompose the hypergraph into at most $k$ subhypergraphs such that each piece satisfies the required regularity conditions (see \pref{alg:hypergraph-decomp-cover} and \pref{obs:small-big-intersect} \& \ref{obs:p-H-bounds}).
Then, there must be one subhypergraph $\calH^{(i)}$ with at least $m/k$ hyperedges, and we will show that there exists an even cover within $\calH^{(i)}$. 

\section{Hypergraph Moore bound for odd arity hypergraphs}
\label{sec:even-cover-odd}

In this section we prove the hypergraph Moore bound for $k$-uniform hypergraphs when $k$ is odd.
\begin{theorem}[\pref{thm:main-thm}, odd $k$]
\label{thm:odd-arity}
    There is a universal constant $B$ such that for any odd $k\in\bbN$, and any $r\in\bbN$ satisfying $2k\le r\le \frac{n}{B^k}$, any $k$-uniform hypergraph $\calH$ with $n$ vertices and $m\ge B^k n\log n\cdot\parens*{\frac{n}{r}}^{k/2-1}$ hyperedges has an even cover of size at most $r\log_2 n$.
\end{theorem}
Our proof strategy broadly involves the following steps.
\begin{itemize}
    \item {\bf Hypergraph decomposition.} We partition $\calH$ into subhypergraphs $\calH^{(0)},\calH^{(1)},\dots,\calH^{(k-1)}$ with the property that every size-$(i+1)$ set in $\calH^{(i)}$ is contained in only a small number of clauses, and every clause in $\calH^{(i)}$ intersects many other clauses at a size-$i$ set.  One of the $\calH^{(i)}$ must contain at least $m/k$ clauses, and we find an even cover in that $\calH^{(i)}$.
    \item {\bf Large $i$.}  When $i\ge\frac{k+1}{2}$, we give a direct reduction to the hypergraph Moore bound for even arity hypergraphs and apply \pref{thm:even-arity}.
    \item {\bf Kikuchi graph.}  To handle the remaining values of $i$, we show the existence of an even cover by proving the contrapositive --- a hypergraph with no small even covers has a bounded number of hyperedges.  To achieve this, we appropriately define the Kikuchi graph for odd arity hypergraphs, and show that the adjacency matrix $\wh{A}$ of some suitably chosen subgraph (via the ``edge deletion process'' described below) satisfies $\wh{A} \preceq Q$ for some diagonal matrix $Q$.  Then the resulting inequality $\1^{\top} \wh{A} \1 \le \tr(Q)$ can be rearranged to bound the number of hyperedges.
    \item {\bf Trace method.}  The way we prove $\wh{A}\psdle Q$ is by using the trace moment method to show $\Norm{Q^{-1/2}\wh{A}Q^{-1/2}}_2 \le 1$.  Bounding a high trace power of $Q^{-1/2}\wh{A}Q^{-1/2}$ corresponds to bounding the total weight of closed walks that use every hyperedge an even number of times in the Kikuchi graph.
    \item {\bf Edge deletion process.}  We delete a small fraction of the edges in $K_r$ with the guarantee that in the resulting subgraph any clause participates in only a small number of incident edges to every vertex.
\end{itemize}

\paragraph{Hypergraph decomposition.}
We describe our algorithm to partition our hypergraph.

\noindent\rule{16cm}{0.4pt}
\begin{algorithm}   \label{alg:hypergraph-decomp-cover}
We partition $\calH$ into hypergraphs $\calH^{(0)},\dots,\calH^{(k-1)}$ via the following algorithm.
\begin{enumerate}
    \item Set $t = k-1$ and $\calH_{\mathrm{current}}\coloneqq \calH$.
    \item Set counter $s = 1$.
    While there is $U\subseteq[n]$ such that $|U|=t$ and $\Abs{\{C\in \calH_{\mathrm{current}} : U\subseteq C\}} \ge \max\braces*{2, \parens*{\frac{n}{r} }^{\frac{k}{2}-t} }$:
    \begin{enumerate}
        \item Choose $U$ satisfying the condition and let $\calH^{(t)}_s$ be a subset of $\{C\in \calH_{\mathrm{current}} : U \subseteq C\}$ of size $\max\braces*{2, \parens*{ \frac{n}{r} }^{\frac{k}{2}-t} }$.
        \item Add all clauses in $\calH^{(t)}_s$ to $\calH^{(t)}$.
        \item Delete all clauses in $\calH^{(t)}_s$ to $\calH_{\mathrm{current}}$.
        \item Increment $s$ by $1$.
    \end{enumerate}
    \item Decrement $t$ by 1.  If $t > 0$, go back to step $2$; otherwise take the remaining clauses in $\calH_{\mathrm{current}}$ and add them to $\calH^{(0)}$.
\end{enumerate}
\noindent\rule{16cm}{0.4pt}
\end{algorithm}

First, observe that the largest subhypergraph $\calH^{(i)}$ in the partition produced by our algorithm must have at least $\frac{m}{k}$ hyperedges.
Next, observe that $i\neq 0$ because if $|\calH^{(0)}| \geq m/k$, then there must be a $j\in[n]$ such that $\Abs{\{C\in \calH^{(0)}: j\in C\}} \geq \frac{m}{nk} \gg (\frac{n}{r})^{k/2-1}$, which would have been added to $\calH^{(1)}$.
Our goal in the rest of the proof is to find a small even cover in $\calH^{(i)}$.  The following observations articulate the properties of $\calH^{(i)}$ we need that are guaranteed by the algorithm.
\begin{observation} \label{obs:p-H-bounds}
    $\calH^{(i)}$ can be partitioned into $\calH^{(i)}_1,\dots,\calH^{(i)}_p$ where for each $j\in[p]$, there is a set $U_j$ of size $i$ such that every $C\in \calH^{(i)}_j$ contains $U_j$, and $\abs{\calH^{(i)}_j} \ge \parens*{\frac{n}{r}}^{\frac{k}{2}-i}$ and $p\le m\cdot\parens*{\frac{r}{n}}^{\frac{k}{2}-i}$.
\end{observation}

\begin{observation} \label{obs:small-big-intersect}
    For $s \geq 1$ and any $U \subseteq [n]$ such that $|U|=i+s$, the number of hyperedges in $\calH^{(i)}$ containing $U$ is at most $\max\Bigset{1,\parens*{\frac{n}{r}}^{\frac{k}{2}-s-i}}$, otherwise they would have been added to $\calH^{(i+s)}$.
\end{observation}

\paragraph{Reduction to even arity case when $i\ge \frac{k+1}{2}$.}  
In this case, by \pref{obs:small-big-intersect}, each pair $C \neq C'$ in any $\calH_j^{(i)}$ must satisfy $C\cap C' = U_j$.
The following makes the reduction from finding even covers in $\calH^{(i)}$ if $i\ge \frac{k+1}{2}$ to the even arity case concrete.
\begin{lemma} \label{lem:large-intersections}
    Let $\calH$ be a $k$-uniform hypergraph on $n$ vertices with no even cover of size $r \log_2 n$.
    Fix $1 \leq i \leq k-1$.
    Suppose $\calH_1,\dots, \calH_p$ are disjoint subsets of $\calH$ such that for each $j \in[p]$, $|\calH_j| \geq 2$ and all pairs of hyperedges $C \neq C'\in \calH_j$ satisfy $C \cap C' = U_j$ for some $U_j \subseteq [n]$ of size $i$.
    Then,
    \begin{equation*}
        \sum_{j=1}^p |\calH_j| \leq O(n\log n) \Paren{\frac{2n}{r}}^{k-i-1} \mper
    \end{equation*}
    In particular, when $i\ge \frac{k+1}{2}$ the above is at most $O(n \log n)\cdot\parens*{\frac{n}{r}}^{k/2-1}$.
\end{lemma}
\begin{proof}
    Given such disjoint subsets $\calH_1,\dots, \calH_p$, we can construct a $2(k-i)$-uniform hypergraph $\wh{\calH}$ by the following:
    for each $j\in[p]$, arbitrarily order the edges: $\calH_j = (C_1,\dots, C_{|\calH_j|})$.
    Then, add the hyperedge $C_s \oplus C_{s+1}$ to $\wh{\calH}$ for $s=1,\dots, |\calH_j|-1$.
    By assumption $|C_s \cap C_{s+1}| = |U_j| = i$, thus $|C_s \oplus C_{s+1}| = 2(k-i)$.
    The resulting $\wh{\calH}$ has
    \begin{equation*}
        |\wh{\calH}| = \sum_{j=1}^p |\calH_j|-1 \geq \frac{1}{2} \sum_{j=1}^p |\calH_j|
    \end{equation*}
    hyperedges, since $|\calH_j| \geq 2$ for all $j \in [p]$.

    We claim that $\wh{\calH}$ cannot have an even cover of size at most $\frac{r}{2} \log_2 n$.
    First, if $\wh{\calH}$ has repeated hyperedges, then there must exist $j \neq j'\in [p]$ and $C_1, C_2\in \calH_j$, $C_1', C_2' \in \calH_{j'}$ such that $C_1 \oplus C_2 = C_1' \oplus C_2'$, but then $\{C_1, C_2, C_1', C_2'\}$ would be an even cover of size 4 in $\calH$.
    Now, suppose $\wh{\calH}$ has no repeated edges but has an even cover of size $\ell$.
    Then, for any $\wh{C}$ in the even cover, we can uniquely identify $j\in[p]$ and $s \leq |\calH_j|-1$ such that $C_s, C_{s+1} \in \calH_j$ and $\wh{C} = C_s \oplus C_{s+1}$.
    Furthermore, by construction there must be at least two $C_s, C_{s'} \in \calH_j$ that each occurs only once.
    Therefore, these edges must form an even cover of size at most $2\ell$ in $\calH$.

    Since $2(k-i)$ is even and $\wh{\calH}$ has no even cover of size $\frac{r}{2}\log_2 n$, we can apply \pref{thm:even-arity} to show that
    \begin{equation*}
        |\wh{\calH}| \leq O(n\log n) \Paren{\frac{2n}{r}}^{k-i-1} \mper
    \end{equation*}
    This completes the proof.
\end{proof}

Henceforth, we assume $i\le\frac{k-1}{2}$, which is the case we need an appropriate Kikuchi graph for odd arity hypergraphs.

\paragraph{Kikuchi matrix for odd arity hypergraphs.} The following is the same Kikuchi graph defined in \cite[Definition 6.2]{GKM21}.

\begin{definition}[Colored Kikuchi graphs and subgraphs] \label{def:odd-kikuchi-graph}
    Fix $r\in \N$ and $t \in \{1,\dots, k-1\}$ such that $2k \leq r \leq n$.
    Let $\calH_1,\dots, \calH_p$ be $p$ disjoint sets of hyperedges such that for each $i\in [p]$, all hyperedges in $\calH_i$ have a common subset $U_i \subset [n]$ where $|U_i| = t$.
    For each $C\in \calH_i$, denote $\wt{C} \coloneqq C\setminus U_i$, and denote $\wt{\calH}_i \coloneqq \{\wt{C}: C \in \calH_i\}$ which can be viewed as a $(k-t)$-uniform hypergraph.
    We define the \emph{colored} Kikuchi graph $K_r$ as follows.

    The vertex set $V(K_r)$ consists of subsets of $[n]\times [2]$ of size $r$, where $S\in V$ is viewed as $(S^{(1)}, S^{(2)})$ where $S^{(1)}, S^{(2)} \subseteq [n]$ are colored \emph{green} and \emph{blue} respectively.
    For each $i\in [p]$ and each $C \neq C' \in \calH_i$, let $\wt{C}^{(1)}$ be $\wt{C}$ colored green and $\wt{C}'^{(2)}$ be $\wt{C}'$ colored blue, and we add an edge between $S, T\in V$, denoted $S \xleftrightarrow{C,C'} T$, if $S \oplus T = \wt{C}^{(1)} \oplus \wt{C}'^{(2)}$ and if one of the following holds,
    \begin{itemize}
        \item $|\wt{C} \cap S^{(1)}| = |\wt{C}' \cap T^{(2)}| = \Ceil{\frac{k-t}{2}}$ and $|\wt{C}' \cap S^{(2)}| = |\wt{C} \cap T^{(1)}| = \Floor{\frac{k-t}{2}}$, or

        \item $|\wt{C} \cap S^{(1)}| = |\wt{C}' \cap T^{(2)}| = \Floor{\frac{k-t}{2}}$ and $|\wt{C}' \cap S^{(2)}| = |\wt{C} \cap T^{(1)}| = \Ceil{\frac{k-t}{2}}$, or
    \end{itemize}
    \pref{fig:kikuchi-example} shows an example of two edges $C,C'\in \calH_i$ forming an edge $(S,T)$ in the Kikuchi graph.

    We say that the edge $(S,T)$ is type-$i$, and for $S\in V$, we define the type-$i$ degree as
    \begin{equation*}
        d_{S,i} \coloneqq \Abs{\Set{C\in \calH_i: |\wt{C} \cap S^{(1)}| \text{ or } |\wt{C} \cap S^{(2)}| \in \Set{\Ceil{\frac{k-t}{2}}, \Floor{\frac{k-t}{2}} } }} \mper
    \end{equation*}
    We call any subgraph of the colored Kikuchi graph as a \emph{colored Kikuchi subgraph}.
\end{definition}

\begin{figure}[ht!]
    \centering
    \includegraphics[width=0.7\textwidth]{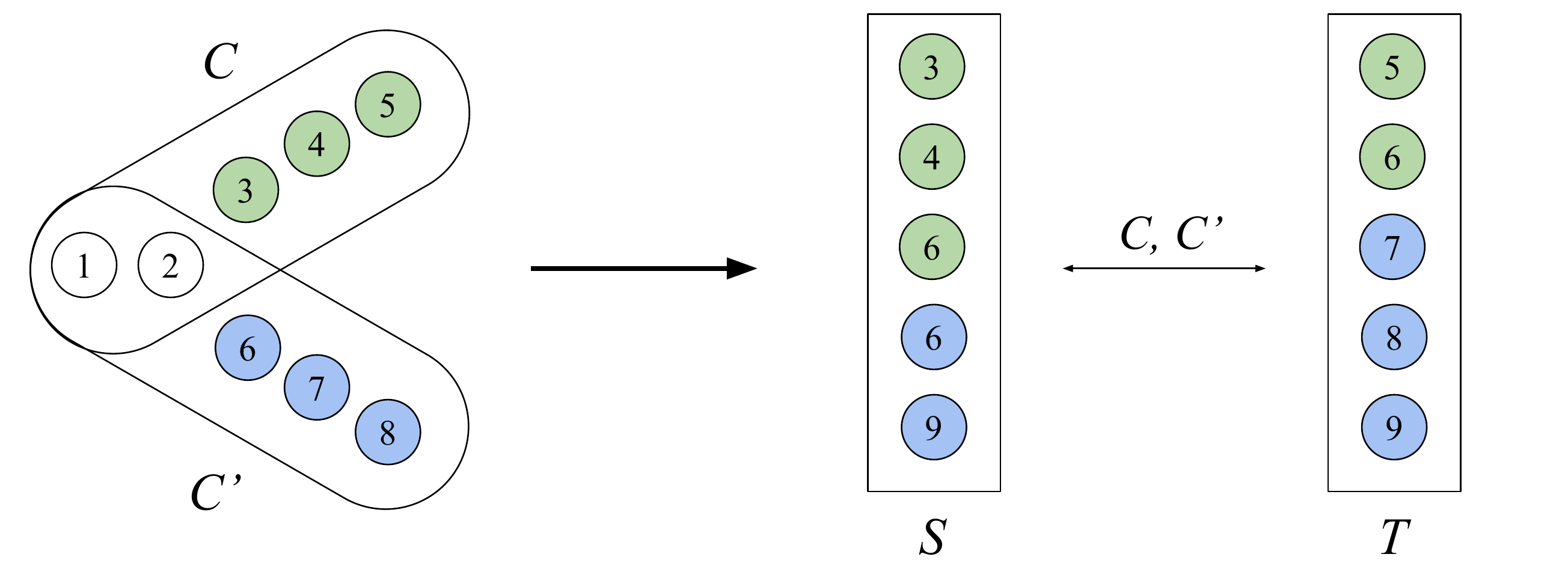}
    \caption{An example of \pref{def:odd-kikuchi-graph} with $k=5$ and $t=2$. On the left are two 5-uniform hyperedges in $\calH_i$ with common intersection $U_i = \{1,2\}$ and $\wt{C} = \{3,4,5\}$, $\wt{C}' = \{6,7,8\}$.
    On the right, $S$ and $T$ are vertices in the Kikuchi graph where $S^{(1)} = \{3,4,6\}$, $T^{(1)} = \{5,6\}$ are colored green,  and $S^{(2)}= \{6,9\}$, $T^{(2)} = \{7,8,9\}$ are colored blue.
    $C$ and $C'$ form an edge between $S,T$ because $|\wt{C} \cap S^{(1)}| = 2$, $|\wt{C} \cap T^{(1)}| = 1$, $|\wt{C}' \cap S^{(2)}| = 1$, and $|\wt{C}' \cap T^{(2)}| = 2$.
    }
    \label{fig:kikuchi-example}
\end{figure}

\begin{remark}[Purpose of coloring]
    The coloring in \pref{def:odd-kikuchi-graph} is needed because $C \neq C'\in \calH_i$ may have intersection larger than $t$, meaning $|C \oplus C'| = |\wt{C} \oplus \wt{C}'| < 2(k-t)$, making the analysis complicated.
    Coloring $\wt{C},\wt{C}'$ with different colors automatically makes $\wt{C}^{(1)}, \wt{C}'^{(2)}$ disjoint, i.e.\ $|S\oplus T| = |\wt{C}^{(1)} \oplus \wt{C}'^{(2)}| = 2(k-t)$.
    Note also that a vertex $S \subseteq [n]\times [2]$ may contain two copies of some element in $[n]$ with different colors, as shown in \pref{fig:kikuchi-example}.
\end{remark}

\begin{observation}[Parameters of the Kikuchi graph]
\label{obs:parameters-of-kikuchi}
    The Kikuchi graph $(V,E)$ defined in \pref{def:odd-kikuchi-graph} has $|V| = \binom{2n}{r}$, and each distinct pair $C,C' \in \calH_i$ contributes a collection of edges $E_{C,C'}$ in $E$, where
    \begin{equation*}
        |E_{C,C'}| = \alpha_t \coloneqq \binom{k-t}{\floor{\frac{k-t}{2}}} \binom{k-t}{\ceil{\frac{k-t}{2}}} \binom{2n-2(k-t)}{r-(k-t)}  \cdot 2^{\1(\text{$k-t$ is odd})}
    \end{equation*}
    by first choosing $\wt{C} \cap S^{(1)}$, $\wt{C}'\cap S^{(2)}$ (or $\wt{C}\cap S^{(2)}$, $\wt{C}'\cap S^{(1)}$) and completing $S$'s remaining $r-(k-t)$ elements.
    Thus, $|E| = \sum_{i=1}^p \binom{|\calH_i|}{2} \cdot \alpha_t$, and standard calculations show that when $2k \leq r \leq n/8$, the average degree $d = \frac{2|E|}{|V|}$ satisfies
    \begin{equation*}
        \Paren{\frac{r}{2n}}^{k-t} \sum_{i=1}^p \binom{|\calH_i|}{2} \leq d \leq 2^{2k} \Paren{\frac{r}{2n}}^{k-t} \sum_{i=1}^p \binom{|\calH_i|}{2} \mper
    \end{equation*}
\end{observation}

Our ideal hope is that the adjacency matrix $A$ of the Kikuchi graph, constructed from $\calH^{(i)} = (\calH^{(i)}_1,\dots,\calH^{(i)}_p)$, is bounded in the PSD order by some low-trace diagonal matrix $Q$.
To achieve this, we prove the following lemma analogous to \pref{lem:even-kikuchi-norm-bound}, but with the additional requirement that $d_{S,i}$ is small for all $S\in V(K_r)$ and $i\in [p]$.
The proof is almost identical to the proof of \pref{lem:even-kikuchi-norm-bound} but the encoding for an ``old hyperedge'' step is different.

\begin{lemma}   \label{lem:good-subgraph-bound}
    Let $r\geq 2k$.
    Given disjoint hyperedges $\calH_1,\dots,\calH_p$, let $\wh{A}$ be the adjacency matrix of any colored Kikuchi subgraph $\wh{K}_r$ as defined in \pref{def:odd-kikuchi-graph}, and let $\Gamma = D + d \Id$ where $D$ is the degree matrix and $d$ is the average degree of $G$.
    Fix $\eta \in \R$ and let $\ell\in \N$ be even.
    Suppose there is no even cover of size at most $\ell$, and suppose $d_{S,i} \leq \eta$ for all $S\in V$ and $i\in[p]$. Then,
    \begin{equation*}
        \Norm{\Gamma^{-1/2}\wh{A} \Gamma^{-1/2}}_2 \leq 2n^{r/\ell}\sqrt{\frac{2\eta\ell}{d}} \mper
    \end{equation*}
\end{lemma}

\begin{proof}
    Let $\wt{A} = \Gamma^{-1/2}\wh{A} \Gamma^{-1/2}$.
    We again use the trace power method:
    \begin{equation*}
        \bignorm{\wt{A}}_2^\ell \leq \tr(\wt{A}^\ell) = \tr( (\Gamma^{-1} A)^\ell) \mper
    \end{equation*}
    Note that each edge $(S,T)$ in $\wh{A}$ corresponds to two hyperedges of the same type (both from some $\calH_i$), one green and one blue, and since there is no even covers of size at most $\ell$, any closed walk must contain an even number of each hyperedge.

    We encode a closed walk $S_1 \rightarrow S_2 \rightarrow \cdots \rightarrow S_{\ell} \rightarrow S_1$ as follows:
    \begin{itemize}
        \item Starting vertex $S_1 \in V$.
        \item One bit $b_i \in \{0,1\}$ at step $i$ to encode whether this step uses two new hyperedges or one (or more) old hyperedge.
        \begin{itemize}
            \item If $b_i = 0$ (two new hyperedges), select one of $S_i$'s neighbors as $S_{i+1}$.
            \item If $b_i = 1$ (old hyperedge), select an old green (or blue) hyperedge $C$ from the previous steps, and select a blue (or green) hyperedge $C'$ incident to $S_i$.
        \end{itemize}
    \end{itemize}
    Recall that for $b\in \{0,1\}$, we write $N_b(S)$ as the possible next steps in the walk from $S$.
    Using the same analysis as the proof of \pref{lem:even-kikuchi-norm-bound}, for $b=0$,
    \begin{equation*}
        \sum_{S_{i+1} \in N_0(S_i)} \frac{1}{d_{S_i}+d} \leq 1 \mcom
    \end{equation*}
    and for $b = 1$, suppose the old edge is of type $j\in [p]$, then $|N_1(S_i)| \leq 2 \ell d_{S_i,j}$ (one previous step, 2 colors), thus
    \begin{equation*}
        \sum_{S_{i+1} \in N_b(S_i)} \frac{1}{d_{S_i}+d} \leq \frac{2\ell d_{S_i,j}}{d_{S_i} + d} \leq \frac{2 \eta\ell}{d} \mper
    \end{equation*}
    We can assume that $2\eta\ell \leq d$, otherwise we can simply treat all steps as new hyperedges.

    There are $\binom{2n}{r} \leq (\frac{2en}{r})^{r} \leq n^r$ (since $r \geq 2k$ and $k\geq 3$) choices to pick the starting vertex $S_1$.
    Furthermore, there can be at most $\ell/2$ steps that use two new hyperedges, i.e.\ $|b| \geq \ell/2$, thus
    \begin{equation*}
        \tr((\Gamma^{-1} \wh{A})^\ell) \leq n^r \sum_{b\in \{0,1\}^{\ell}}\Paren{\frac{2\eta\ell}{d}}^{|b|} 
        \leq 2^{\ell} n^r \Paren{\frac{2\eta\ell}{d}}^{\ell/2} \mper
    \end{equation*}
    Taking the $\ell$-th root completes the proof.
\end{proof}



\paragraph{Construction of colored Kikuchi subgraph.}
Unfortunately, the requirement for all $d_{S,i}$ to be bounded by a small $\eta$ prohibits us from obtaining a good bound on the adjacency matrix of the full colored Kikuchi graph $K_r$ using \pref{lem:good-subgraph-bound}.  This motivates dropping a small number of edges from $K_r$, and bounding the adjacency matrix $\wh{A}$ of the resulting subgraph $\wh{K}_r$ instead.
Thus, we proceed with identifying a suitable colored Kikuchi subgraph $\wh{K}_r$ of $\calH^{(i)}$ with adjacency matrix $\wh{A}$ via the following \emph{edge deletion process}:
\begin{displayquote}
    Start with the colored Kikuchi graph $K_r$, and delete every edge $\{S,T\}$ caused by a pair of clauses $C, C'$ such that $S$ or $T$ has strictly more than $1$ edge that $C$ or $C'$ participates in.
\end{displayquote}
To obtain a handle on the average degree of $\wh{K}_r$, we first show that the number of edges of $K_r$ we delete to obtain $\wh{K}_r$ is only a small fraction of the total number of edges, and then the desired lower bound follows from a lower bound on $|E(K_r)|$.

\paragraph{Analyzing the edge deletion process.}
We find it convenient to think of the fraction of deleted edges as the \emph{probability that a uniformly random edge in $K_r$ is absent in $\wh{K}_r$}.
With this probabilistic interpretation in hand, observe that a uniformly random edge in $K_r$ is the same as choosing a uniformly random pair of clauses $(C,C')$ such that $C$ and $C'$ both belong to the same $\calH_j^{(i)}$ and then choosing a random edge $\{S,T\}$ in $E_{C,C'}$, the collection of edges adorned by $(C,C')$.
We will use the notation $C''\arrow_C S$ to mean $\abs{\wt{C}''\cap S} = \abs{\wt{C} \cap S}$, where we recall from \pref{def:odd-kikuchi-graph} that $\wt{C} \coloneqq C\setminus U_j$ with $U_j$ being the size-$i$ common intersection of $\calH_j^{(i)}$.
We then show the following.
\begin{claim}[Deletion probability]
\label{claim:few-deletions}
    For every pair of clauses $(C,C')$ such that $C$ and $C'$ belong to the same $\calH_j^{(i)}$ for some $j\in [p]$,
    \[
        \Pr_{\{S,T\}\sim E_{C,C'}}\bracks*{\{S,T\}\text{ deleted}} \le k\cdot 4^{k+1} \sqrt{\frac{r}{n}} \mper
    \]
\end{claim}
\begin{proof}
Recall that we defined $\wt{C} = C\setminus U_j$ and $\wt{C}' = C'\setminus U_j$.
The distribution of $S = (S^{(1)}, S^{(2)})$ (the green and blue vertices) is uniform on all sets such that:
\begin{itemize}
    \item $\abs{\wt{C}\cap S^{(1)}} = \Ceil{\frac{k-i}{2}}$, $\abs{\wt{C}'\cap S^{(2)}} = \Floor{\frac{k-i}{2}}$, or
    \item $\abs{\wt{C}\cap S^{(1)}} = \Ceil{\frac{k-i}{2}}$, $\abs{\wt{C}'\cap S^{(2)}} = \Floor{\frac{k-i}{2}}$.
\end{itemize}
Then, by union bound,
\begin{align*}
    \Pr_{\{S,T\}\sim E_{C,C'}} \bracks*{ \{S,T\}\text{ deleted} } \le& \Pr_{\{S,T\}\sim E_{C,C'}} \bracks*{ \exists C''\arrow_C S^{(1)}: C''\in \calH_j^{(i)}, C'' \ne C } +\\
    &\Pr_{\{S,T\}\sim E_{C,C'}} \bracks*{ \exists C'' \arrow_{C'} S^{(2)}: C''\in \calH_j^{(i)}, C'' \ne C' } +\\
    &\Pr_{\{S,T\}\sim E_{C,C'}} \bracks*{ \exists C'' \arrow_C T^{(1)}: C''\in \calH_j^{(i)}, C'' \ne C }+\\
    &\Pr_{\{S,T\}\sim E_{C,C'}} \bracks*{ \exists C'' \arrow_{C'} T^{(2)}: C''\in \calH_j^{(i)}, C'' \ne C' } \\
    =&~ 4 \Pr_{\{S,T\}\sim E_{C,C'}} \bracks*{ \exists C''\arrow_C S^{(1)}: C''\in \calH_j^{(i)}, C'' \ne C} \\
    \intertext{then by Markov's inequality,}
    \le&~4\underset{\{S,T\}\sim E_{C,C'}}{\E} \Abs{C'':C''\arrow_C S^{(1)}, C''\in\calH^{(i)}_j, C''\ne C} \\
    =&~4 \sum_{\substack{C'':C''\in \calH_j^{(i)} \\ C''\ne C }} \Pr_{\{S,T\}\sim E_{C,C'}} \bracks*{C''\arrow_C S^{(1)}} \numberthis \label{eq:prob-delete}
\end{align*}
Once the intersection of $S$ with $\wt{C}$ and $\wt{C}'$ is chosen, the remaining elements are selected uniformly at random without replacement.
For fixed $C''\ne C \in \calH^{(i)}_j$, since they contain $U_j$ of size $i$, $|\wt{C}'' \cap \wt{C}| = |C'' \cap C| - i$, and $S$ must include $\floor{\frac{k-i}{2}} - (|C''\cap C|-i)$ additional    elements from $\wt{C}'' \setminus \wt{C}$ for $C'' \to_C S^{(1)}$ to hold.
Thus,
\[
    \Pr_{\{S,T\}\sim E_{C,C'}} \bracks*{C''\to_C S^{(1)}} \le 2^k \parens*{\frac{r}{n}}^{\Floor{\frac{k-i}{2}} - |C''\cap C|+i}.
\]
Thus, we can prove:
\begin{align*}
    \pref{eq:prob-delete} &\le 4\cdot 2^k \sum_{s=i}^{k-1} \sum_{\substack{U\subseteq C \\ |U|=s}} \sum_{\substack{C'':C''\in\calH_j^{(i)}\\ C'' \ne C \\ C''\cap C = U }} \parens*{\frac{r}{n}}^{\Floor{\frac{k-i}{2}} - s + i}
    \numberthis \label{eqn:deletion-rate-bound}
\end{align*}
By \pref{obs:small-big-intersect}, we can bound the above as
\begin{align*}
    &\le 4\cdot 2^k \sum_{s=i}^{k-1} \sum_{\substack{U\subseteq C \\ |U|=s}} \parens*{\frac{n}{r}}^{\frac{k}{2}-s} \parens*{\frac{r}{n}}^{\frac{k-i}{2} - \frac{\Ind[k-i\text{ odd}]}{2} - s + i} \\
    &\le k\cdot 4^{k+1} \sqrt{\frac{r}{n}} \mcom
\end{align*}
as $\frac{i}{2} - \frac{\Ind[k-i\text{ odd}]}{2} \geq \frac{1}{2}$ for all $i\geq 1$ when $k$ is odd.
\end{proof}
\paragraph{Lower bound on average degree in $A$.}
By choosing $B$ large enough, the upper bound on $r$, and \pref{claim:few-deletions}, the fraction of edges we delete from the original colored Kikuchi graph $K_r$ to obtain $\wh{K}_r$ is at most $.5$ and hence $d(\wh{K}_r)\ge .5d(K_r)$ where $d(K_r)$ and $d(\wh{K}_r)$ are the average degrees in $K_r$ and $\wh{K}_r$ respectively.  Thus, we know:
\[
    d(K_r) \ge \parens*{\frac{r}{2n}}^{k-i} \sum_{j=1}^p {|\calH^{(i)}_j|\choose 2} \ge \parens*{\frac{r}{2n}}^{k-i} \cdot p \cdot { m/kp \choose 2} \ge \parens*{\frac{r}{2n}}^{k-i} \cdot \frac{m^2}{4k^2 p}
\]
where the first inequality uses \pref{obs:parameters-of-kikuchi}, and the second inequality is due to Jensen's inequality.

By the upper bound $p\le m\cdot\parens*{\frac{r}{n}}^{\frac{k}{2}-i}$ as noted in \pref{obs:p-H-bounds}:
\[
    d(K_r) \ge \frac{1}{4k^2 2^k}\cdot\parens*{\frac{r}{n}}^{k-i}\cdot\parens*{\frac{n}{r}}^{\frac{k}{2}-i}\cdot m = \frac{1}{4k^2 2^k}\cdot \parens*{\frac{r}{n}}^{\frac{k}{2}} \cdot m.
\]
As an upshot, we know:
\begin{claim}   \label{claim:avg-deg-lb}
    $\displaystyle d(\wh{K}_r)\ge \frac{1}{8k^2 2^k} \cdot \parens*{\frac{r}{n}}^{k/2} \cdot m$.
\end{claim}

\paragraph{Spectral double counting.}
With a lower bound on $d(\wh{K}_r)$ in hand, we are now ready to perform our weighted spectral double counting argument to complete the proof of \pref{thm:odd-arity}.
\begin{proof}[Proof of \pref{thm:odd-arity}]
Recall that our goal is to prove that there is a small even cover in $\calH^{(i)}$, the largest piece obtained from the decomposition, and also recall that if $i\ge\frac{k+1}{2}$, then we are done by \pref{lem:large-intersections}.  Hence, we assume $i\le\frac{k-1}{2}$ for the rest of the proof.

Suppose there are no even covers in $\calH$ of size $\ell = r \log n$, then there are also none in $\calH^{(i)}$ from \pref{lem:good-subgraph-bound} we get:
\[
    \Norm{\Gamma^{-1/2} \wh{A} \Gamma^{-1/2}}_2 \le 4\sqrt{\frac{2\ell}{d(\wh{K}_r)}}.
\]
Thus, $\wh{A}\psdle 4\sqrt{\frac{2\ell}{d(\wh{K}_r)}}\Gamma$, and by taking the quadratic form with the all-ones vector, we get:
\[
    2|E(\wh{K}_r)| = \1^\top \wh{A}\1 \le 4\sqrt{\frac{2\ell}{d(\wh{K}_r)}}\cdot\tr(\Gamma) = 16\sqrt{\frac{2\ell}{d(\wh{K}_r)}}\cdot |E(\wh{K}_r)|,
\]
which implies
\[
    d(\wh{K}_r) \le 128\ell,
\]
and by our lower bound on $d(\wh{K}_r)$ from \pref{claim:avg-deg-lb}, we get
\[
    \frac{1}{8k^22^k}\cdot\parens*{\frac{r}{n}}^{\frac{k}{2}}\cdot m \le 128 r\log n,
\]
which we can rearrange as
\[
    m \le B^k n\log n \cdot \parens*{\frac{n}{r}}^{k/2-1}.
\]
for some large enough constant $B$.
Thus, if $m$ is lower bounded as in the theorem statement, there must be an even cover of size $\ell\log n$.
\end{proof}
\section{Strong refutation of semirandom \texorpdfstring{$k$}{k}-XOR}
\label{sec:refutation}

In this section, we show that our reweighted Kikuchi matrix and edge deletion process yield a significantly simpler analysis of strong refutation algorithms for semirandom $k$-XOR formulas and lose only a single $\log n$ factor in the density.
Combined with Feige's ``XOR principle''~\cite{Fei02,AOW15}, we also obtain refutation algorithms for all \emph{smoothed} Boolean CSPs. We will omit such reduction in this work and direct the reader to~\cite{GKM21} for a detailed exposition.

\begin{theorem}[Semirandom $k$-XOR refutation]
\label{thm:refutation-main-thm}
    Fix $k\in \N$.
    There is an algorithm with parameter $r\in \N$, $2k \leq r \leq n/8$ that takes as input a semirandom $k$-XOR instance
    \begin{equation*}
        \psi(x) = \frac{1}{m}\sum_{C\in \calH} b_C x_C
    \end{equation*}
    where $\calH$ is a $k$-uniform hypergraph with $n$ vertices and $m$ hyperedges, and each $b_C \in \pmo$ is chosen uniformly at random.
    The algorithm has the following guarantee: there is a universal constant $C$ such that if $m \geq C^k n\log n \cdot (\frac{n}{r})^{\frac{k}{2}-1} \eps^{-4}$ for $\eps \in (0,1/2)$, then with probability over $1-\frac{1}{\poly(n)}$ over $\{b_C\}_{C\in\calH}$,
    the algorithm runs in time $n^{O(r)}$ and certifies that $\psi(x) \leq \eps$.
\end{theorem}

\begin{remark}[Refutation strength: dependence on $\eps$]
    For the even arity case, we actually obtain a stronger guarantee (weaker requirement) of $m \geq O(n\log n) \cdot (\frac{n}{r})^{\frac{k}{2}-1} \eps^{-2}$.
    For the odd arity case however, our analysis incurs a (likely suboptimal) dependence of $1/\eps^4$ on the refutation strength (i.e., the upper bound on the value of the input $k$-XOR instance), though improving the $1/\eps^5$ dependence of~\cite[Theorem 5.1]{GKM21}.
    In contrast, a $1/\eps^2$ dependence is known to hold for fully random $k$-XOR instances~\cite{RRS17}.
    Apart from a somewhat unsatisfying deficiency, this suboptimality turns out to be consequential -- in particular, it changes the threshold at which efficient FKO refutation witnesses exist for semirandom $k$-SAT (and other CSPs) by a polynomial factor in $n$.
    Finding the ``right'' dependence of $1/\eps^2$ (for the odd case) is an interesting open problem.
\end{remark}


Our refutation algorithm will utilize the same Kikuchi graphs from \pref{def:even-kikuchi-graph} and \pref{def:odd-kikuchi-graph} but with signs added to the edges in the natural way.

\begin{definition}[Signed Kikuchi graph]
\label{def:signed-kikuchi-graph}
    Let $\calH$ be a $k$-uniform hypergraph associated with $\pmo$ signs $\{b_C\}_{C\in\calH}$.
    For the even arity case, let $A_b$ be the signed adjacency matrix of the Kikuchi graph from \pref{def:even-kikuchi-graph} where each edge $S \xleftrightarrow{C} T$ has a sign $b_C$.
    For the odd arity case, let $A_b$ be the signed adjacency matrix of the Kikuchi graph from \pref{def:odd-kikuchi-graph} where each edge $S \xleftrightarrow{C,C'} T$ has a sign $b_C b_{C'}$.
\end{definition}

\subsection{Refuting semirandom even arity XOR}

In this section, we prove \pref{thm:refutation-main-thm} when $k$ is even.
As we will see in the short proof, our idea of the reweighted Kikuchi matrix from the hypergraph Moore bound naturally applies here, and in fact, we obtain the ``right'' $1/\eps^2$ dependence in this case, i.e., we can certify that $\psi(x) \leq \eps$ when $m \geq O(n\log n) \cdot (\frac{n}{r})^{\frac{k}{2}-1} \eps^{-2}$.


Recall that in the Kikuchi graph $(V,E)$, each $C\in \calH$ contributes $\alpha \coloneqq \frac{1}{2} \binom{k}{k/2} \binom{n-k}{r-k/2}$ edges in $E$, hence $|E| = \frac{1}{2}|V| d = m\alpha$.
Thus, it is clear that
\begin{equation*}
    \psi(x) = \frac{1}{m} \cdot \frac{1}{\alpha} \sum_{(S,T) \in E} b_{S \oplus T} x_{S \oplus T}
    = \frac{1}{\binom{n}{r}d}(x^{\odot r})^{\top} A_b x^{\odot r}
    \numberthis
    \label{eq:even-refutation}
\end{equation*}
where $x^{\odot r} \in \pmo^{\binom{n}{r}}$ and the $S$-entry of $x^{\odot r}$ is $x_S$ for $S\subseteq [n]$, $|S|=r$.

We now follow the same reweighting strategy: with $\Gamma = D + d\Id$, we bound the spectral norm of the reweighted Kikuchi matrix $\Norm{\Gamma^{-1/2} A_b \Gamma^{-1/2}}_2$ with an almost identical proof as \pref{lem:even-kikuchi-norm-bound}.

\begin{lemma}
\label{lem:even-signed-kikuchi-norm-bound}
    Let $k$ be even and $r\in \N$.
    Let $A_b$ be the signed Kikuchi graph with random $\pmo$ coefficients $\{b_C\}_{C\in \calH}$, and let $\Gamma = D + d\Id$ where $D$ is the degree matrix and $d$ is the average degree of the Kikuchi graph.
    Then, with probability at least $1-\frac{1}{\poly(n)}$ over the randomness of $\{b_C\}_{C\in \calH}$,
    \begin{equation*}
        \Norm{\Gamma^{-1/2}A_b \Gamma^{-1/2}}_2 \leq O\Paren{\sqrt{\frac{r\log n}{d}}} \mper
    \end{equation*}
\end{lemma}
\begin{proof}
    Let $\wt{A}_b = \Gamma^{-1/2} A_b \Gamma^{-1/2}$.
    We again use the trace power method $\norm{\wt{A}_b}_2^\ell \leq \tr((\Gamma^{-1} A_b)^\ell)$ where we choose an even $\ell = 2\ceil{r\log_2 n}$.
    Observe that in expectation, $\E_b\tr((\Gamma^{-1} A_b)^\ell)$ counts the closed walks that use each hyperedge an even number of times.
    This is exactly the same as \pref{lem:even-kikuchi-norm-bound} where we count closed walks in an unsigned Kikuchi graph assuming there is no even cover of size $\leq \ell$.
    Thus, \pref{lem:even-kikuchi-norm-bound} shows that
    \begin{equation*}
        \E_b\tr((\Gamma^{-1} A_b)^\ell) \leq 2^\ell n^r \Paren{\frac{\ell}{d}}^{\ell/2} \leq O\Paren{\frac{\ell}{d}}^{\ell/2}
    \end{equation*}
    when $\ell \geq r\log_2 n$.
    Then, by Markov's inequality, for any $\lambda > 0$,
    \begin{equation*}
        \Pr_b\Brac{\norm{\wt{A}_b}_2 \geq \lambda}
        = \Pr_b\Brac{\norm{\wt{A}_b}_2^{\ell} \geq \lambda^\ell} \leq \lambda^{-\ell} \cdot \E_b \tr((\Gamma^{-1} A_b)^\ell)
        \leq O\Paren{\frac{\ell}{\lambda^2 d}}^{\ell/2}
    \end{equation*}
    Choosing $\lambda = O(\sqrt{\ell/d})$ completes the proof.
\end{proof}

We can complete the proof of \pref{thm:refutation-main-thm} for even $k$.
\begin{proof}[Proof of \pref{thm:refutation-main-thm} for even $k$]
    Let $A_b$ be the signed Kikuchi graph with signs $\{b_C\}_{C\in \calH}$, let $\Gamma = D + d\Id$ where $D$ is the degree matrix and $d$ is the average degree of the Kikuchi graph, and let $\wt{A}_b = \Gamma^{-1/2} A_b \Gamma^{-1/2}$.
    The certification algorithm is simply to compute $\|\wt{A}_b\|_2$.
    Since $A_b \preceq \norm{\wt{A}_b}_2\cdot \Gamma$, and $\tr(\Gamma)= 2\binom{n}{r}d$, by \pref{lem:even-signed-kikuchi-norm-bound},
    \begin{equation*}
        \psi(x) = \pref{eq:even-refutation} \leq \frac{1}{\binom{n}{r}d} \norm{\wt{A}_b}_2 \cdot \tr(\Gamma) \leq O\Paren{\sqrt{\frac{r\log n}{d}}}
    \end{equation*}
    using the fact that $x^{\odot r} \in \pmo^{\binom{n}{r}}$ and $(x^{\odot r})^\top \Gamma x^{\odot r} = \tr(\Gamma)$.
    There is some constant $C$ such that when $m \geq C n\log n \cdot (\frac{n}{r})^{\frac{k}{2}-1}\eps^{-2}$,
    by \pref{eqn:even-average-degree} the average degree $d \geq \frac{1}{2}(\frac{r}{n})^{k/2}m = \frac{C}{2}r\log n \cdot \eps^{-2}$, thus giving us $\psi(x) \leq \eps$.
    This completes the proof.
\end{proof}

\subsection{Refuting semirandom odd arity XOR}


Our proof of \pref{thm:refutation-main-thm} for the odd arity case closely mimics the steps taken in proving the hypergraph Moore bound for odd arity hypergraphs (\pref{thm:odd-arity}).
Given a semirandom $k$-XOR instance $\psi$ on hypergraph $\calH$ with random signs $\{b_C\}_{C\in \calH}$, we first apply the following hypergraph decomposition algorithm (a variant of \pref{alg:hypergraph-decomp-cover}) to decompose the hypergraph into subhypergraphs $\calH^{(1)},\dots,\calH^{(k-1)}$.
The main difference compared to \pref{alg:hypergraph-decomp-cover} is that in the final step, we add the ``leftover'' hyperedges to $\calH^{(1)}$ instead of an extra $\calH^{(0)}$.

\noindent\rule{16cm}{0.4pt}
\begin{algorithm}[Hypergraph decomposition]
\label{alg:hypergraph-decomposition-refutation}
    Given a $k$-uniform hypergraph $\calH$ on $n$ vertices and $m$ hyperedges, and thresholds $\tau_1,\dots,\tau_{k-1} \geq 2$, we partition $\calH$ into hypergraphs $\calH^{(1)},\dots,\calH^{(k-1)}$ via the following algorithm.
    \begin{enumerate}
        \item Set $t = k-1$ and $\calH_{\mathrm{current}}\coloneqq \calH$.
        \item Set counter $s = 1$.
        While there is $T\subseteq[n]$ such that $|T|=t$ and $\Abs{\{C\in \calH_{\mathrm{current}} : T\subseteq C\}} \ge \tau_t$:
        \begin{enumerate}
            \item Choose $T$ satisfying the condition and let $\calH^{(t)}_s$ be a subset of $\{C\in \calH_{\mathrm{current}} : T\subseteq C\}$ of size $\tau_t$.
            \item Add all clauses in $\calH^{(t)}_s$ to $\calH^{(t)}$.
            \item Delete all clauses in $\calH^{(t)}_s$ to $\calH_{\mathrm{current}}$.
            \item Increment $s$ by $1$.
        \end{enumerate}
        \item Decrement $t$ by 1.  If $t > 0$, go back to step $2$; otherwise take the remaining clauses in $\calH_{\mathrm{current}}$ and partition them into $n$ parts $F_1,\dots,F_n$ where each clause $C$ goes to some $F_i$ such that $i\in C$.
        Add $F_1,\dots,F_n$ to $\calH^{(1)}$ and terminate.
    \end{enumerate}
\end{algorithm}
\noindent\rule{16cm}{0.4pt}

\paragraph{Notations and parameters.}
Throughout this section we will use the following notations.
\begin{itemize}
    \item In \pref{alg:hypergraph-decomposition-refutation}, we set thresholds $\tau_t = \max\braces*{1, \parens*{\frac{n}{r} }^{\frac{k}{2}-t}} \cdot 4k \eps^{-2}$.

    \item In the decomposition, each $\calH^{(t)}$ contains $p_t$ groups $\calH_{1}^{(t)},\dots, \calH_{p_t}^{(t)}$ where group $\calH_i^{(t)}$ has a center $T_i^{(t)}$ of size $t$,
    and for each $C\in \calH^{(t)}_i$, we write $\wt{C} = C\setminus T_i^{(t)}$.

    \item Each $|\calH_i^{(t)}| = \tau_t$, with the exception that $|\calH_i^{(1)}| \leq \tau_1$ may have different sizes (the leftover hyperedges in \pref{alg:hypergraph-decomposition-refutation}).
    Let $m_t \coloneqq \sum_{i=1}^{p_t} |\calH_i^{(t)}|$ be the total number of hyperedges in $\calH^{(t)}$.

    \item When $t=1$ and $m \geq C^k n\log n \cdot (\frac{n}{r})^{\frac{k}{2}-1}\eps^{-4}$ for a large enough constant $C$, we have $m \geq n \tau_1$, hence $p_1 \leq \frac{m}{\tau_1} + n \leq \frac{2m}{\tau_1}$.
    Thus, we will use $p_t \tau_t \leq 2m$ for all $t\in[k-1]$.

    \item For each $t\in[k-1]$, the colored Kikuchi graph $(V,E)$ obtained from $\calH^{(t)} = (\calH_1^{(t)},\dots, \calH_{p_t}^{(t)})$ (from \pref{def:odd-kikuchi-graph}) has edges $|E| = \alpha_t \sum_{i=1}^{p_t}\binom{|\calH_i^{(t)}|}{2} \leq \frac{1}{2}\alpha_t m_t \tau_t$, where $\alpha_t \approx (\frac{2n}{r})^{r-(k-t)}$ is the number of edges contributed by each distinct pair $C,C'\in \calH_i$ (see \pref{obs:parameters-of-kikuchi}).
\end{itemize}

\noindent With these notations and parameters in mind, we can write $\psi(x)$ as
\begin{align*}
    \psi(x) &= \frac{1}{m} \sum_{t=1}^{k-1} \sum_{C\in \calH^{(t)}} b_C x_C
    = \frac{1}{k}\sum_{t=1}^{k-1} \psi_t(x) \\
    \quad \text{where} \quad
    \psi_t(x) &\coloneqq
    \frac{k}{m}\sum_{i=1}^{p_t} \sum_{C\in \calH_i^{(t)}} b_C x_C = \frac{k}{m} \sum_{i=1}^{p_t} x_{T_i} \sum_{C \in \calH_i^{(t)}} b_{C} x_{\wt{C}} \mper
    \numberthis \label{eqn:psi-t}
\end{align*}
Essentially, each $\psi_t$ is the sub-instance of $\psi$ restricted to the partition $\calH^{(t)}$.
Recall that for the purpose of showing existence of even covers, we only need to focus on one $\calH^{(t)}$.
For refutation however, we need to certify a bound on $\psi_t(x)$ for all $t \in [k-1]$.

\begin{lemma}[Refuting each $\psi_t$] \label{lem:refuting-psi-t}
    Fix an odd $k\in \N$, $t \in [k-1]$, and let $2k \leq r \leq n/8$.
    There is a constant $C$ such that given a semirandom $k$-XOR instance $\psi$ with $n$ variables and $m \geq C^k n\log n (\frac{n}{r})^{\frac{k}{2}-1} \eps^{-4}$ clauses for $\eps \in (0,1/2)$, and suppose $\psi_t$ is the subinstance from \pref{eqn:psi-t} obtained by the hypergraph decomposition algorithm (\pref{alg:hypergraph-decomposition-refutation}), then with probability $1-\frac{1}{\poly(n)}$ over the random signs, we can certify that $\psi_t(x) \leq \eps$ in $n^{O(r)}$ time.
\end{lemma}

\pref{lem:refuting-psi-t} immediately completes the proof of \pref{thm:refutation-main-thm} for odd $k$.

\begin{proof}[Proof of \pref{thm:refutation-main-thm} by \pref{lem:refuting-psi-t}]
    Given the hypergraph $\calH$, we apply the hypergraph decomposition algorithm (\pref{alg:hypergraph-decomposition-refutation}) with thresholds $\tau_1,\dots, \tau_{k-1}$ and obtain subinstances $\psi_1,\dots, \psi_{k-1}$ as in \pref{eqn:psi-t}.
    For each $t\in[k-1]$, we can certify that $\psi_t(x) \leq \eps$ by \pref{lem:refuting-psi-t} with high probability, which immediately implies the desired bound $\psi(x) \leq \eps$.
\end{proof}

\paragraph{Edge deletion process.}
The proof of \pref{lem:refuting-psi-t} requires deleting the ``bad''  edges from the signed Kikuchi matrix $A_b^{(t)}$ via a similar deletion process as the one used in the proof of \pref{thm:odd-arity}, but with some parameter $\eta > 1$ instead of $1$ and an additional \emph{equalizing} step:
\begin{displayquote}
    Start with the colored Kikuchi graph, and delete every edge $\{S,T\}$ caused by a pair of clauses $C, C'\in \calH_i^{(t)}$ such that $S$ or $T$ has more than $\eta$ edges that $C$ or $C'$ participates in.

    Suppose $\rho < 1$ is the maximum fraction of edges deleted among all pairs of clauses. Then, for every $i\in[p_t]$ and every distinct pair $C,C'\in \calH^{(t)}_i$, we delete (additional) edges caused by $C,C'$ arbitrarily such that exactly $\rho$ fraction of edges are deleted.
\end{displayquote}

\begin{observation}[Uniform deletion] \label{obs:uniform-deletion}
    The final step in the above edge deletion process ensures that every pair $C,C'$ contributes the \emph{same} number of edges ($(1-\rho)\alpha_t$ to be exact) in the Kikuchi graph.
\end{observation}

Mirroring the proof of \pref{claim:few-deletions} yields the following generalization.

\begin{lemma}[Deletion rate]
\label{lem:deletion-rate}
    Suppose a subhypergraph $\calH^{(i)}$ satisfies that for any $s \geq i$ and any $T\subseteq [n]$ with $|T| = s$, the number of hyperedges in $\calH^{(i)}$ containing $T$ is at most $\tau_s$, then the deletion process with parameter $\eta \geq 1$ satisfies
    \begin{equation*}
        \Pr_{\{S,T\}\sim E_{C,C'}} \bracks*{ \{S,T\}\textnormal{ deleted} } \leq \frac{4^k}{\eta} \cdot \sum_{s=i}^{\floor{\frac{k+i}{2}}} \tau_s \Paren{\frac{r}{n}}^{\floor{\frac{k+i}{2}} - s} \mper
    \end{equation*}
\end{lemma}
\begin{proof}
    The proof is identical to the proof of \pref{claim:few-deletions}.
    \hyperref[eq:prob-delete]{Eq.~\textup{(\ref*{eq:prob-delete})}} holds with an additional $1/\eta$ factor due to Markov's inequality.
    The lemma statement then follows immediately from \pref{eqn:deletion-rate-bound}.
\end{proof}

\paragraph{Proof of \pref{lem:refuting-psi-t} via the Cauchy-Schwarz trick and the deletion process.}

\begin{proof}[Proof of \pref{lem:refuting-psi-t}]
    We apply the Cauchy-Schwarz trick to $\psi_t$ from \pref{eqn:psi-t}:
    \begin{align*}
        \psi_t(x)^2 &\leq \frac{k}{m^2} \sum_{i=1}^{p_t} x_{T_i}^2 \cdot \sum_{i=1}^{p_t} \Paren{\sum_{C\in\calH_i^{(t)}} b_{C} x_{\wt{C}}}^2
        \leq \frac{k p_t}{m^2} \sum_{i=1}^{p_t} \sum_{C,C'
        \in \calH_i^{(t)}} b_{C} b_{C'} x_{\wt{C}} x_{\wt{C}'} \\
        &\leq \frac{k p_t m_t}{m^2} +  \frac{k p_t}{m^2} \sum_{i=1}^{p_t} \sum_{C\neq C' \in \calH_i^{(t)}} b_{C} b_{C'} x_{C \oplus C'}
        \numberthis \label{eqn:psi-t-squared}
    \end{align*}
    since $x\in \pmo^n$, $b_C\in \pmo$ and $\sum_{i=1}^{p_t} |\calH_i^{(t)}| = m_t$.
    For the first term, since for all $t\in[k-1]$, we set $\tau_t \geq 4k\eps^{-2}$ and $p_t \leq 2m/\tau_t \leq \frac{m \eps^2}{2k}$, thus
    \begin{equation*}
        \frac{kp_t m_t}{m^2} \leq \frac{\eps^2}{2} \mper
        \numberthis \label{eqn:psi-t-first-term}
    \end{equation*}
    We can now focus our attention on the second term in \pref{eqn:psi-t-squared}.

    Given $\calH^{(t)}$ and its partitions $\calH_1^{(t)},\dots, \calH_{p_t}^{(t)}$ of size $\tau_t$, and signs $\{b_C\}_{C\in \calH^{(t)}}$, let $A_b^{(t)}$ be the signed Kikuchi matrix defined in \pref{def:signed-kikuchi-graph}, which is the signed version of the colored Kikuchi graph $(V,E)$ from \pref{def:odd-kikuchi-graph}.
    Recall from \pref{obs:parameters-of-kikuchi} that each distinct pair $C,C' \in \calH_i^{(t)}$ contributes $\alpha_t \approx (\frac{2n}{r})^{r-(k-t)}$ edges in the graph.
    Thus, similar to \pref{eq:even-refutation} in the even case, we can write the second term of \pref{eqn:psi-t-squared} as a quadratic form:
    \begin{equation*}
        f_t(x) \coloneqq \frac{k p_t}{m^2} \sum_{i=1}^{p_t} \sum_{C\neq C' \in \calH_i^{(t)}} b_C b_{C'} x_{C \oplus C'}
        = \frac{k p_t}{2\alpha_t m^2} (x^{\odot r})^\top A_b^{(t)} x^{\odot r} \numberthis \label{eqn:ft-quadratic-form}
    \end{equation*}
    where $x^{\odot r}\in \pmo^{\binom{2n}{r}}$ such that for $S\in [n]\times [2]$ with $S = (S^{(1)}, S^{(2)})$ (green and blue elements), the $S$-entry of $x^{\odot r}$ is $x_{S^{(1)} \oplus S^{(2)}}$.

    We proceed to certify an upper bound on $f_t(x)$.
    Given the signed Kikuchi matrix $A_b^{(t)}$, we first apply the deletion process with parameter $\eta = B^k \eps^{-2}$ for some large enough constant $B$.
    With the chosen thresholds $\tau_s$, \pref{lem:deletion-rate} states that the deletion probability $\rho$ is at most
    \begin{equation*}
        \rho \leq \frac{4^k}{\eta} \cdot \sum_{s=t}^{\floor{\frac{k+t}{2}}} 4k\eps^{-2} \cdot \max\braces*{1, \parens*{\frac{n}{r} }^{\frac{k}{2}-s} } \cdot \Paren{\frac{r}{n}}^{\floor{\frac{k+t}{2}} - s}
        \leq \frac{1}{2} \mcom
    \end{equation*}
    since $s \leq \floor{\frac{k+t}{2}}$ in the summation and $\floor{\frac{k+t}{2}} \geq \frac{k+1}{2}$ for all $t\geq 1$.

    Let $\wh{A}_b^{(t)}$ be the Kikuchi matrix after the deletion process.
    By \pref{obs:uniform-deletion}, each distinct pair $C,C'\in \calH_i^{(t)}$ contributes exactly $(1-\rho)$ fraction of the original edges.
    Thus, we have
    \begin{equation*}
        (x^{\odot r})^\top \wh{A}_b^{(t)} x^{\odot r} =
        (1-\rho) \cdot (x^{\odot r})^\top A_b^{(t)} x^{\odot r} \mper
        \numberthis \label{eqn:A-hat-quadratic-form}
    \end{equation*}

    Next, we follow the same argument as the proof of \pref{lem:even-signed-kikuchi-norm-bound} to analyze $\wh{A}_b^{(t)}$, using the norm bound of \pref{lem:good-subgraph-bound}.
    Let $\Gamma = D + d\Id$ where $D$ is the degree matrix and $d$ is the average degree, and let $\wt{A}_b = \Gamma^{-1/2} \wh{A}_b^{(t)} \Gamma^{-1/2}$.
    To bound $\norm{\wt{A}_b}_2$, we again use the trace power method $\norm{\wt{A}_b}_2^\ell \leq \tr((\Gamma^{-1} \wh{A}_b^{(t)})^\ell)$ where we choose an even $\ell = 2\ceil{r\log_2 n}$.
    Observe that in expectation, $\E_b\tr((\Gamma^{-1} A_b)^\ell)$ counts the closed walks that use each hyperedge an even number of times.
    This is exactly the same as \pref{lem:good-subgraph-bound} where we count closed walks in an unsigned Kikuchi graph assuming there is no even cover of size $\leq \ell$.
    Furthermore, $d_{S,i} \leq \eta$ is automatically satisfied after the deletion process.
    Thus, we can directly apply \pref{lem:good-subgraph-bound} and show that
    \begin{equation*}
        \E_b\tr\Paren{(\Gamma^{-1} \wh{A}_b^{(t)})^\ell} \leq 2^\ell n^r \Paren{\frac{2\eta\ell}{d}}^{\ell/2} \leq O\Paren{\frac{\eta\ell}{d}}^{\ell/2}
    \end{equation*}
    when $\ell \geq r\log_2 n$.
    Then, by Markov's inequality, we have that $\Pr_b\Brac{\norm{\wt{A}_b}_2 \geq O\Paren{ \sqrt{\frac{\eta\ell}{d}} }} \leq \frac{1}{\poly(n)}$.

    Thus, with high probability we have $\wh{A}_b^{(t)} \preceq O\Paren{ \sqrt{\frac{\eta\ell}{d}} } \cdot \Gamma$, then since $\tr(\Gamma) = 4|E|$,
    \begin{equation*}
        (x^{\odot r})^\top \wh{A}_b^{(t)} x^{\odot r} \leq O\Paren{\sqrt{\frac{\eta\ell}{d}}}\cdot \tr(\Gamma)
        = O\Paren{\sqrt{\frac{\eta\ell}{d}}}\cdot |E| \mper
    \end{equation*}

    Next, let $\wh{f}_t(x) = \frac{k p_t}{2\alpha_t m^2} (x^{\odot r})^\top \wh{A}_b^{(t)} x^{\odot r}$.
    By \pref{obs:parameters-of-kikuchi}, we have $d \geq (\frac{r}{2n})^{k-t} \sum_{i=1}^{p_t} \binom{|\calH_i^{(t)}|}{2}$ when $2k \leq r \leq n/8$.
    Plugging in parameters $|E| = \alpha_t \sum_{i=1}^{p_t} \binom{|\calH_i^{(t)}|}{2}$, $p_t \tau_t \leq 2m$,
    $\eta = B^k \eps^{-2}$, and $\ell = 2\ceil{r\log_2 n}$, standard calculations show that
    \begin{equation*}
        \wh{f}_t(x) \leq O(1)\frac{k p_t}{\alpha_t m^2} \sqrt{\frac{\eta \ell}{d}} |E|
        \leq O(1) \frac{kp_t}{m^2} \sqrt{\eta \ell \Paren{\frac{2n}{r}}^{k-t} \sum_{i=1}^{p_t}\binom{|\calH_i^{(t)}}{2}}
        \leq O(1) \sqrt{\frac{\eta r\log n}{m\tau_t} \Paren{\frac{2n}{r}}^{k-t}} \mper
    \end{equation*}
    Suppose $m \geq C^k n\log n \cdot (\frac{n}{r})^{\frac{k}{2}-1} \eps^{-4}$ for some large enough constant $C$.
    We split into cases:
    \begin{enumerate}
        \item For $t\leq \frac{k-1}{2}$, we set $\tau_t = (\frac{n}{r})^{\frac{k}{2}-t} \cdot 4k\eps^{-2}$, thus $\wh{f}_t(x) \leq \frac{\eps^2}{4}$.

        \item For $t \geq \frac{k+1}{2}$, we set $\tau_t = 4k\eps^{-2}$, thus $\wh{f}_t(x) \leq \frac{\eps^2}{4}(\frac{n}{r})^{\frac{k}{4}-\frac{t}{2}} < \frac{\eps^2}{4}$.
    \end{enumerate}
    Therefore, by calculating $\norm{\wt{A}_b}_2$, which can be done in $n^{O(r)}$ time, we can certify that $\wh{f}_t(x) \leq \frac{\eps^2}{4}$.
    Combined with \pref{eqn:A-hat-quadratic-form} and the bound of $\rho\leq 1/2$, we can certify that
    \begin{equation*}
        f_t(x) \leq \frac{1}{1-\rho} \cdot \wh{f}_t(x) \leq \frac{\eps^2}{2} \mcom
    \end{equation*}
    and with \pref{eqn:psi-t-first-term}, we can certify an upper bound on \pref{eqn:psi-t-squared}:
    \begin{equation*}
        \psi_t(x)^2 \leq (\ref{eqn:psi-t-first-term}) + (\ref{eqn:ft-quadratic-form}) \leq \frac{\eps^2}{2} + f_t(x)
        \leq \eps^2 \mcom
    \end{equation*}
    completing the proof.
\end{proof}

\ifnum\RemoveAuthor=1
{}
\else
\section*{Acknowledgments}
We would like to thank Omar Alrabiah for heroic feedback on an earlier draft of this paper and Benny Sudakov for an illuminating discussion and comments.
J.H.\ would like to thank Peter Manohar for discussions about~\cite{GKM21}.
S.M.\ would like to thank Siqi Liu and Tselil Schramm for discussions on an independent problem that fueled some of the ideas.
Finally, much of this work was conducted when the authors were visiting the program on ``Computational Complexity of Statistical Inference'' at Simons Institute for the Theory of Computing in Fall 2021.  We would like to thank the institute for their hospitality and support!
S.M.\ was also visiting Microsoft Research Redmond when part of this work was done.

\fi

\bibliographystyle{alpha}
\bibliography{main}

\appendix
\section{Alternative proof of the Moore bound for irregular graphs}
\label{app:exact-moore-bound}

We proved the weak Moore bound (\pref{prop:weak-moore-bound}) by showing that if there is no cycle of length $\leq \ell$, then $A \prec \frac{2n^{1/\ell}}{\sqrt{d}}(D+d\Id)$ (\pref{claim:weak-moore-bound-norm}) where $D$ is the diagonal degree matrix and $d$ is the average degree, which then gives us a bound of $2\ceil{\log_{d/16} n}$.
In this section, we prove that using a more carefully chosen diagonal matrix $\Gamma'$, such a strategy can recover the exact Moore bound $2 \log_{d-1} n$.
This provides an alternative proof of the Moore bound in addition to the existing proofs by~\cite{AHL02} and \cite{BR14}\footnote{\cite{AHL02} and \cite{BR14} actually obtained a slightly more precise bound depending on whether the girth of the graph is odd or even.}.

\begin{theorem}[Moore bound for irregular graphs] \label{thm:moore-bound}
    Suppose $G$ is a graph on $n$ vertices with average degree $d > 2$.  Then $G$ has a cycle of length $2(\floor{\log_{d-1} n} + 1)$.
\end{theorem}

The following lemma shows what the ``correct'' diagonal matrix should be to recover the exact Moore bound.

\begin{lemma}
\label{lem:A-Ihara-bass}
    Let $G$ be a graph with $n$ vertices and degree matrix $D$ that has no cycle of length $\leq \ell$ for some even $\ell\in\N$.
    Then, the adjacency matrix $A$ satisfies
    \begin{equation*}
        A \preceq n^{2/\ell} \Id + n^{-2/\ell} (D-\Id) \mper
    \end{equation*}
\end{lemma}

\begin{proof}[Proof of \pref{thm:moore-bound} by \pref{lem:A-Ihara-bass}]
    Assuming there is no cycle of length $\leq \ell$, \pref{lem:A-Ihara-bass} implies that
    \begin{equation*}
        \vec{1}^\top A \vec{1} = nd \leq n\cdot (n^{2/\ell} + n^{-2/\ell}(d-1)) \mper
    \end{equation*}
    Let $x = n^{2/\ell}$, then we have $x^2 - dx + (d-1) \geq 0$, which implies that $x \geq d-1$ (as $x\leq 1$ is not valid).
    Taking logs, we get
    \begin{equation*}
        \frac{2}{\ell}\log n \geq \log(d-1)
        \Longrightarrow \ell \leq 2\log_{d-1} n \mper
    \end{equation*}
    $\ell$ is even, so $\ell < 2(\floor{\log_{d-1}n}+1)$. This completes the proof.
\end{proof}

The proof of \pref{lem:A-Ihara-bass} is based on \emph{non-backtracking walks}, which are walks such that no edge is the inverse of its preceding edge.
We note that both proofs of \cite{AHL02} and \cite{BR14} also analyze non-backtracking walks.
For a graph $G$ on $n$ vertices with adjacency matrix $A$, we define $A^{(s)}$ to be the $n\times n$ matrix whose $(u, v)$ entry counts the number of length-$s$ non-backtracking walks between vertices $u$ and $v$ in $G$.
The following is a standard fact.
\begin{fact}[Recurrence and generating function of $A^{(s)}$]
\label{fact:non-backtracking}
    The non-backtracking matrices $A^{(s)}$ satisfy the following recurrence:
    \begin{equation*}
    \begin{aligned}
        A^{(0)} &= \Id \mcom \\
        A^{(1)} &= A \mcom \\
        A^{(2)} &= A^2 - D \mcom \\
        A^{(s)} &= A^{(s-1)} A - A^{(s-2)} (D-\Id) \mcom \quad s > 2 \mper
    \end{aligned}
    \end{equation*}
    The recurrences imply that these matrices have a generating function:
    \begin{equation*}
        J(t) \coloneqq \sum_{s=0}^{\infty} A^{(s)} t^{s} = (1-t^2) \cdot H(t)^{-1} \mcom
        \text{ where } H(t) \coloneqq \Id - At + (D-\Id)t^2
    \end{equation*}
    for $t \in [0, 1)$ whenever the series converges.
\end{fact}

We first prove the following lemma,
\begin{lemma} \label{lem:upper-bound-l-walks}
    Let $s,k\in \N$, $s\geq k$, and let $q,r$ be the quotient and remainder of $s$ divided by $k$, i.e.\ $s = qk + r$.  Then,
    \begin{equation*}
        \tr(A^{(s)}) \leq \sqrt{n} \cdot \norm{A^{(k)}}_2^{q} \cdot \norm{A^{(r)}}_F.
    \end{equation*}
\end{lemma}
\begin{proof}
    $\tr(A^{(s)})$ counts the number of closed non-backtracking walks of length $s$ in the graph.
    Now, consider the set of closed walks of length $s= qk + r$ such that after every $k$ non-backtracking steps, we can ``forget the previous step'', i.e.\ we are allowed to backtrack at step $ik$ for every $i = 0,\dots, q$.
    The number of such walks is $\tr((A^{(k)})^q A^{(r)})$.
    The set of closed non-backtracking walk is clearly a subset of such walks, thus we have
    \begin{equation*}
        \tr(A^{(s)}) \leq \tr( (A^{(k)})^q A^{(r)})
        \leq \Norm{(A^{(k)})^q}_F \cdot \Norm{A^{(r)}}_F.
    \end{equation*}
    Let $\lambda_1,\dots,\lambda_n$ be the eigenvalues of $A^{(k)}$ and $\lambda_{\max} = \|A^{(k)}\|_2$. Then,
    \begin{equation*}
        \Norm{(A^{(k)})^q}_F = \sqrt{\sum_{i=1}^n \lambda_i^{2q} } \leq \sqrt{n} (\lambda_{\max})^{q}.
    \end{equation*}
    This completes the proof.
\end{proof}

With \pref{fact:non-backtracking} and \pref{lem:upper-bound-l-walks}, we now prove \pref{lem:A-Ihara-bass} by analyzing the convergence of $J(t)$ as $t$ increases from $0$.

\begin{proof}[Proof of \pref{lem:A-Ihara-bass}]
    Let $A$ be the adjacency matrix of $G$ with average degree $d>2$, and let $D$ be the diagonal degree matrix $G$.
    Recall the definitions $J(t) = \sum_{s=0}^\infty A^{(s)} t^s$ and $H(t) = \Id - At + (D-\Id)t^2$ from \pref{fact:non-backtracking}.
    We will analyze the convergence of $\tr(J(t))$ as $t$ increase from $0$.

    Observe that $J(0) = H(0) = \Id$, and since $J(t)$ and $H(t)$ are both symmetric matrices, their eigenvalues move continuously on the real line as $t$ increases from 0.
    Thus, suppose there is some $t^* \in (0,1)$ such that $\tr(J(t)) < \infty$ for all $t\in [0, t^*)$, then $H(t) \succ 0$ for all $t\in[0,t^*)$.
    This is easy to see because if not, then there must be some $t'\in[0,t^*)$ such that $H(t')\succeq 0$ but has a zero eigenvalue, and $\tr(J(t'))$ will not converge.

    We next show that we can take $t^* = n^{-2/\ell}$ assuming that $G$ has no cycle of length $\leq \ell=2k$.
    First, observe that every entry of $A^{(k)}$ must be either 0 or 1, otherwise if $A^{(k)}[i,j] > 1$ then there are two distinct length-$k$ paths from $i$ to $j$, meaning there is a cycle of length at most $2k=\ell$, a contradiction.
    Therefore, the $L_1$ norm of each row of $A^{(k)}$ is at most $n$, hence $\|A^{(k)}\|_2 \leq n$.
    Next, observe that for each $s\in\N$ we can write $s = qk+r$, and
    \begin{equation*}
        J(t) = \sum_{s=0}^{\infty} A^{(s)} t^{s}
        \leq \sum_{r=0}^{k-1} \sum_{q=0}^{\infty} A^{(qk+r)} t^{qk+r}.
    \end{equation*}
    By \pref{lem:upper-bound-l-walks}, we have
    \begin{equation*}
        \tr(J(t)) \leq \sum_{r=0}^{k-1} t^r \sqrt{n} \|A^{(r)}\|_F \sum_{q=0}^\infty \|A^{(k)}\|_2^q \cdot t^{qk}
        \leq \sum_{r=0}^{k-1} t^r \sqrt{n} \|A^{(r)}\|_F \sum_{q=0}^\infty (nt^k)^q.
    \end{equation*}
    Thus, if $t < n^{-1/k} < 1$, then $\tr(J(t)) < \infty$.
    Therefore, we have $H(t) \succ 0$ for all $t\in [0, n^{-1/k})$, and by continuity $H(n^{-1/k})\succeq 0$, which means that
    \begin{equation*}
        \Id - n^{-1/k} A + n^{-2/k} (D-\Id) \succeq 0
        \Longrightarrow A \preceq n^{2/\ell} \Id + n^{-2/\ell} (D-\Id)
    \end{equation*}
    as $\ell = 2k$.
    This completes the proof.
\end{proof}





\end{document}